\documentclass[11pt, a4paper]{amsart}

\usepackage{amssymb,amsmath,amsthm}
\usepackage{tikz-cd}
\usetikzlibrary{arrows}
\usepackage{graphicx}
\usepackage[hidelinks]{hyperref}
\usepackage{mathtools}
\usepackage[titletoc,title]{appendix}
\usepackage{graphicx}
\usepackage{caption}
\usepackage{subcaption}
\usepackage{blindtext}
\usepackage[titletoc,title]{appendix}
\usepackage[english]{babel}
\usepackage{mathrsfs}
\usepackage{mathabx}
\usepackage{cite}
\usepackage{mathabx}
\usepackage{microtype}
\usepackage[capitalize]{cleveref}

\newtheorem{theorem}{Theorem}[section]
\newtheorem{lemma}[theorem]{Lemma}
\newtheorem{corollary}[theorem]{Corollary}
\newtheorem{proposition}[theorem]{Proposition}
\theoremstyle{definition}
\newtheorem{definition}[theorem]{Definition}

\newtheorem{remark}[theorem]{Remark}

\newtheorem{conjecture}[theorem]{Conjecture}
\newtheorem*{theorem*}{Theorem}
\newtheorem*{conjecture*}{Conjecture}

\newtheorem{claim}[theorem]{Claim}

\newcommand{\isom}{\cong}

\newcommand{\N}{\mathbb{N}}
\newcommand{\F}{\mathbb{F}}
\newcommand{\Z}{\mathbb{Z}}

\makeatletter
\def\Ddots{\mathinner{\mkern1mu\raise\p@
\vbox{\kern7\p@\hbox{.}}\mkern2mu
\raise4\p@\hbox{.}\mkern2mu\raise7\p@\hbox{.}\mkern1mu}}

\makeatother

\newcommand{\normal}[1]{\langle\!\langle #1 \rangle\!\rangle}

\newcommand{\wh}[1]{\widehat{ #1 }}
\newcommand{\wt}[1]{\widetilde{ #1 }}

\def\immerses{\looparrowright}
\def\injects{\hookrightarrow}

\DeclareSymbolFontAlphabet{\amsmathbb}{AMSb}

\DeclareMathOperator{\rk}{rk}
\DeclareMathOperator{\rr}{rr}
\DeclareMathOperator{\len}{\ell}

\DeclareMathOperator{\core}{Core}

\DeclareMathOperator{\Out}{Out}
\DeclareMathOperator{\MCG}{MCG}
\DeclareMathOperator{\stab}{Stab}

\DeclareMathOperator{\bs}{BS}

\DeclarePairedDelimiter\abs{\lvert}{\rvert}

\makeatletter
\let\oldabs\abs
\def\abs{\@ifstar{\oldabs}{\oldabs*}}

\newcounter{cases}
\newcounter{subcases}[cases]

\tikzset{% from the manual
math to/.tip={Glyph[glyph math command=rightarrow]},
loop/.tip={Glyph[glyph math command=looparrowleft, swap]},
loop'/.tip={Glyph[glyph math command=looparrowleft]},
 weird/.tip={Glyph[glyph math command=Rrightarrow, glyph length=1.5ex]},
  pi/.tip={Glyph[glyph math command=pi, glyph length=1.5ex, glyph axis=0pt]},
}

\newcounter{marcocomments}

\begin{document}

\title{The geometry of subgroups of mapping tori of free groups}

\author{Marco Linton}
\address{Instituto de Ciencias Matem\'aticas, CSIC-UAM-UC3M-UCM, Madrid, Spain}
\email{marco.linton@icmat.es}

\begin{abstract}
We show that finitely generated mapping tori of free groups have a canonical collection of maximal sub-mapping tori of finitely generated free groups with respect to which they are relatively hyperbolic and locally relatively quasi-convex. As a consequence, we characterise locally quasi-convex hyperbolic groups amongst free-by-cyclic and one-relator groups. We also upgrade several known results for mapping tori of finitely generated free groups to the general case, such as the computations of Dehn functions, the solution to the conjugacy problem and the characterisation of the finitely generated intersection property.
\end{abstract}

\maketitle

\section{Introduction}

Let $\F$ be a free group and let $\psi\colon \F\to \F$ be a monomorphism. Then the \emph{mapping torus} of $\psi$ is the group $M(\psi)$ with presentation:
\[
M(\psi) = \langle \F, t \mid t^{-1}ft = \psi(f), \,\forall f\in \F\rangle.
\]
Very few properties are known for (subgroups of) mapping tori $M(\psi)$ without strong restrictions on $\F$ and $\psi$. Feighn--Handel showed in \cite{FH99} that $M(\psi)$ is coherent ---that is, finitely generated subgroups are finitely presented. A result of Borisov--Sapir \cite{BS05} combined with a result of Chong--Wise \cite{CW24} implies that every finitely generated subgroup of $M(\psi)$ is residually finite (see also the earlier result of Baumslag for free-by-cyclic groups \cite{Ba71}). Finally, Mutanguha \cite{Mu21} showed that a finitely generated subgroup of $M(\psi)$ is hyperbolic if and only if it does not contain a Baumslag--Solitar subgroup. The aim of this article is to add one more property to this short list.

There are two special subclasses of mapping tori that are worth discussing before we state our main theorem. The first is the class of \emph{\{fg free\}-by-cyclic groups} ---these are the mapping tori $M(\psi)$ where $\F = F_n$ is finitely generated and $\psi$ is an isomorphism. This subclass is particularly interesting for its close connections with the study of 3-manifolds. Many results proven for the mapping class group $\MCG(S)$ and fibred 3-manifold have analogues in the study of $\Out(F_n)$ and \{fg free\}-by-cyclic groups, albeit with additional challenges. The main tool used to study \{fg free\}-by-cyclic groups are train tracks, introduced by Bestvina--Handel \cite{BH92}. These are self maps of graphs with particularly nice properties which have been heavily exploited to connect dynamical properties of automorphisms $\psi\in\Out(F_n)$ with geometric, algebraic and algorithmic properties of the mapping torus $M(\psi) = F_n\rtimes_{\psi}\Z$. For example, Brinkmann \cite{Br00} showed that $M(\psi)$ is hyperbolic precisely when $\psi$ has no periodic conjugacy classes and Ghosh \cite{Gh23} showed that $G$ is relatively hyperbolic precisely when $\psi$ is exponentially growing. See \cite{BFH00,BFH05,BMMV06,BG10,HW15} for many more examples.

The other subclass is that of \emph{free-by-cyclic groups} ---these are the mapping tori $M(\psi)$ where $\psi$ is an isomorphism. Many well-studied classes of groups have been recently shown, somewhat surprisingly, to actually be virtually free-by-cyclic: this includes 3-manifold groups of rational cohomological dimension 2 \cite{KL24a}, one-relator groups with torsion or negative immersions \cite{KL24b}, coherent uniform lattices in Bourdon buildings \cite{KL24b} and generic groups of deficiency at least 2 \cite{KKW22}. Amongst virtually RFRS groups there is also a very useful homological characterisation of virtually free-by-cyclic groups due to Fisher \cite{Fi25}. Moreover, it was conjectured by Abdenbi--Wise \cite[Conjecture 1.6]{AW23} that all infinite locally quasi-convex hyperbolic groups are actually virtually mapping tori of free groups (and hence, by Fisher's criterion, virtually free-by-cyclic).

In this article we are interested in the geometry of finitely generated subgroups of arbitrary finitely generated mapping tori of free groups. Our main theorem identifies a canonical relatively hyperbolic structure on a finitely generated mapping torus $M(\psi)$ with respect to which $M(\psi)$ is locally relatively quasi-convex. Recall that a relatively hyperbolic group pair $(G, \mathcal{P})$ is \emph{locally relatively quasi-convex} if all finitely generated subgroups $H\leqslant G$ are relatively quasi-convex with respect to $\mathcal{P}$ (we follow Hruska \cite{Hr10} for the relevant definitions, see \cref{sec:rel_hyp}).

\begin{theorem}
\label{main}
Let $\F$ be a free group, $\psi\colon \F\to \F$ a monomorphism and let $G \leqslant M(\psi)$ be a finitely generated non-free subgroup of the mapping torus. 

There is a (possibly empty) canonical finite collection of (conjugacy classes of) subgroups $\mathcal{P}$ of $G$, each isomorphic to a mapping torus of a finitely generated free group, with the following properties:
\begin{enumerate}
\item If $H\leqslant G$ is isomorphic to a mapping torus of a finitely generated non-trivial free group, then $H$ is conjugate within $G$ into a unique $P\in\mathcal{P}$.
\item $(G, \mathcal{P})$ is a relatively hyperbolic pair.
\item $(G, \mathcal{P})$ is a locally relatively quasi-convex pair.
\end{enumerate}
\end{theorem}

Local relative quasi-convexity is a strong property which is particularly useful for promoting properties of subgroups of the peripherals to properties of subgroups of the whole group. For instance, Dahmani showed in \cite{Da03} that limit groups are locally relatively quasi-convex (with respect to their maximal non-cyclic abelian subgroups) and used this to show that they have the finitely generated intersection property (also known as the Howson property). Dahmani's theorem was then generalised by Bigdely--Wise in \cite{BW13}. In \cref{sec:applications} we use \cref{main} to promote known results on mapping tori of finitely generated free groups to arbitrary finitely generated mapping tori of free groups. Specifically, if $M(\psi)$ is finitely generated, we show the following:
\begin{enumerate}
\item We identify the possible Dehn functions $M(\psi)$ can have in \cref{dehn_function} (the case in which $\F = F_n$ was handled by Mutanguha \cite{Mu24}).
\item We show that $M(\psi)$ has decidable conjugacy problem in \cref{conjugacy_problem} (the case in which $\F = F_n$ was handled by Logan \cite{Lo23}).
\item We characterise when $M(\psi)$ has the finitely generated intersection property in \cref{fgip} (the case in which $\F = F_n$ was handled by Bamberger--Wise \cite{BW22}).
\item We characterise when $M(\psi)$ has all its finitely generated subgroups undistorted in \cref{thm:locally_undistorted}.
\end{enumerate}

It is a difficult open problem whether mapping tori of finitely generated free groups have decidable membership problem and whether they are effectively coherent ---that is, whether there is an algorithm which, on input a finite subset $S\subset G$, computes a finite presentation for the subgroup $\langle S\rangle$---, see \cite{Ka00,Aim23}. Carstensen showed in \cite{Ca22} that relative quasi-convexity constants for finitely generated subgroups of torsion-free locally relatively quasi-convex groups and generating sets for their induced peripherals can be computed, provided the peripherals have decidable membership problem. In particular, if the subgroups $\mathcal{P}$ in \cref{main} have decidable membership problem and are effectively coherent, then $G$ also has decidable membership problem and is effectively coherent.

\subsection{Quasi-convex subgroups of mapping tori}

In the case in which the peripheral collection $\mathcal{P}$ from \cref{main} is empty, the group $G$ is hyperbolic and locally quasi-convex. Conversely, if $G$ is hyperbolic, but the peripheral collection $\mathcal{P}$ from \cref{main} is non-empty, then the base groups of the peripheral mapping tori are not quasi-convex as they have infinite height (see work of Gitik--Mitra--Rips--Sageev \cite{GMRS98}). This leads us to a characterisation of hyperbolic and locally quasi-convex mapping tori of free groups, solving a problem of Wise \cite[Problem 17]{Wi20} and a more general conjecture of Abdenbi--Wise \cite[Conjecture 6.1]{AW23}.

\begin{corollary}
\label{cor:lqh}
The following are equivalent for a finitely generated mapping torus of a free group $M(\psi)$:
\begin{enumerate}
\item $M(\psi)$ is hyperbolic and locally quasi-convex.
\item $M(\psi)$ contains no subgroup isomorphic to a mapping torus of a finitely generated non-trivial free group.
\item $\rk(\F) = \infty$ and $\psi$ is \emph{fully irreducible}: that is, there is no proper finitely generated free factor $1\neq F\leqslant \F$ so that $\psi^m(F)$ is conjugate into $F$ for some $m\geqslant 1$.
\end{enumerate}
\end{corollary}

A large source of examples of mapping tori satisfying the conclusions of \cref{cor:lqh} is provided by a result of Mutanguha. If $\F$ is finitely generated and $\psi\colon \F\to \F$ is fully irreducible and atoroidal, then Mutanguha showed in \cite{Mu20} that no infinite index subgroup of $M(\psi)$ is a mapping torus of a finitely generated free group. Thus, every finitely generated infinite index subgroup of $M(\psi)$ is locally quasi-convex hyperbolic by \cref{cor:lqh}.

A natural problem that arises now is to determine precisely which finitely generated subgroups of hyperbolic mapping tori of free groups are quasi-convex. A class of groups in which this problem has a satisfying solution is the class of fundamental groups of hyperbolic 3-manifolds. Precisely, the \emph{subgroup tameness theorem} ---which is a consequence of the solution of the Tameness conjecture by Calegari--Gabai \cite{CG06} and Agol \cite{Ag04}, Canary's covering theorem \cite{Ca96} and a result of Hruska's \cite[Corollary 1.6]{Hr10}-- states: if $M^3$ is a hyperbolic 3-manifold and $H\leqslant \pi_1(M^3)$ is a finitely generated subgroup, then $H$ is either a virtual fibre subgroup (i.e. the fundamental group of a surface fibre of a fibration of a finite sheeted cover of $M^3$ over the circle), or is a (relatively) quasi-convex subgroup (with respect to the maximal parabolics). For mapping tori of free groups $M(\psi)$, the analogue of a surface fibre is a finitely generated subgroup $H\leqslant M(\psi)$ so that $M(\psi)\cong M(\phi)$ with $\phi\colon H\to H$ a monomorphism. If $\phi$ is an isomorphism, $H$ is a \emph{fibre subgroup}, otherwise it is a \emph{semi-fibre subgroup}. The following appears to be the correct conjectural analogue of the subgroup tameness theorem for mapping tori of free groups, first conceived at an AIM workshop \cite{Aim23} and also posed by Abdenbi--Wise \cite[Problem 1.5]{AW23}.

\begin{conjecture}
\label{conj}
Suppose that $\F$ is a finitely generated free group and $\psi\colon \F\to \F$ is a fully irreducible monomorphism so that $G = M(\psi)$ is hyperbolic. If $H\leqslant G$ is a finitely generated subgroup, then one of the following holds:
\begin{enumerate}
\item $H$ is a fibre or semi-fibre subgroup of a finite index subgroup of $G$.
\item $H$ is quasi-convex.
\end{enumerate}
\end{conjecture}

\cref{cor:lqh} implies \cref{conj} if $\F$ is instead assumed to be infinitely generated. When $\psi$ is surjective, but not atoroidal, $M(\psi)$ is the fundamental group of a fibred hyperbolic 3-manifold by a result of Bestvina--Handel \cite{BH92} and so the conjecture holds if $G$ is instead assumed to not be hyperbolic and if quasi-convexity is replaced with relative quasi-convexity. When $\psi$ is not surjective, then $M(\psi)$ is hyperbolic by a result of Mutanguha \cite{Mu20}. 

Note that the assumption of $\psi$ being fully irreducible cannot be dropped in \cref{conj} as Brady--Dison--Riley \cite{BDR13} provided examples of hyperbolic \{fg free\}-by-cyclic groups with reducible monodromy which contain finitely generated subgroups with distortion function not bounded by any finite tower of exponentials, whereas a fibre and a semi-fibre subgroup must have exponential distortion.

\subsection{One-relator groups}

Using \cref{cor:lqh} we may also completely characterise when a one-relator group is locally quasi-convex hyperbolic. The reader is directed to the survey article \cite{CFL25} for history, background and recent progress in the theory of one-relator groups. Our characterisation will be in terms of the \emph{primitivity rank} $\pi(w)$ of the relator $w\in F$. This is defined as:
\[
\pi(w) = \min\{\rk(H) \mid w\in H\leqslant F, \, w \text{ not primitive in } H\}\in \N\cup \{\infty\}.
\]
The primitivity rank $\pi(w)$, introduced by Puder \cite{Pu14}, has recently been shown to have strong connections with geometric and subgroup properties of the one-relator group $F/\normal{w}$, see the work of Louder--Wilton \cite{LW22,LW24} and \cite{Li25} for some examples. In order to apply \cref{cor:lqh}, we use the fact that if $\pi(w)\neq 2$, then $F/\normal{w}$ is virtually free-by-cyclic, proved in \cite{KL24b}.

\begin{theorem}
\label{thm:one-relator}
If $G = F/\normal{w}$ is a finitely generated one-relator group, then the following are equivalent:
\begin{enumerate}
\item $G$ is locally quasi-convex hyperbolic.
\item $\pi(w)\neq 2$.
\end{enumerate}
\end{theorem}

Note that Puder provided an algorithm to compute the primitivity rank $\pi(w)$ \cite{Pu14} and so consequently there is also an algorithm to decide whether a one-relator group is locally quasi-convex hyperbolic. Previously, McCammond--Wise \cite{MW05} and Hruska--Wise \cite{HW01} had proven that $F/\normal{w^n}$ is locally quasi-convex hyperbolic when $n$ is sufficiently large. \cref{thm:one-relator} implies that we only need to take $n\geqslant 2$.

Louder--Wilton showed in \cite{LW24} that presentation complexes of one-relator groups $F/\normal{w}$ with $\pi(w)> 2$ satisfy a type of combinatorial negative curvature called negative irreducible curvature (see \cite[Theorem 10.7]{Wi24}). Wilton conjectured \cite[Conjecture 12.9]{Wi24} that all compact 2-complexes with this property should have locally quasi-convex hyperbolic fundamental group. \cref{thm:one-relator} therefore solves an important special case of this conjecture. We also point out that, using a result of Abdenbi--Wise \cite{AW23}, one could also add a fourth equivalent condition to \cref{cor:lqh} in terms of negative irreducible curvature.

Although one-relator groups are known to be coherent \cite{JZL25}, it is an open problem as to whether they are all effectively coherent \cite[Problem 2.5.33]{CFL25}. Since locally quasi-convex hyperbolic groups are effectively coherent (see work of Grunschlag \cite[Proposition 6.1]{Gr99}), \cref{thm:one-relator} implies that many one-relator groups are also effectively coherent. The following corollary answers a question of Louder--Wilton \cite[Question 6.7]{LW24}.

\begin{corollary}
If $G = F/\normal{w}$ is a one-relator group with $\pi(w)\neq 2$, then $G$ is effectively coherent.
\end{corollary}

Haglund--Wise \cite{HW08} showed that locally quasi-convex hyperbolic groups that are virtually compact special are LERF ---that is, all finitely generated subgroups are separable. Since the groups from \cref{thm:one-relator} are known to be virtually compact special by work of Wise \cite{Wi21} and \cite{Li25}, we obtain the following corollary, answering a question of Fine--Rosenberger \cite[Question OR9]{FR01} and providing many new examples of LERF groups.

\begin{corollary}
If $G = F/\normal{w}$ is a finitely generated one-relator group with $\pi(w)\neq 2$, then $G$ is LERF.
\end{corollary}

\subsection{Summary of the article}

After a section of preliminaries, \cref{sec:preliminaries}, in \cref{sec:criteria} we lay the ground work for the proof of \cref{main} by establishing a criterion for relative quasi-convexity of subgroups of graphs of relatively hyperbolic groups, possibly of independent interest. Our criterion, \cref{thm:qc_criterion}, states that if a relatively hyperbolic group $G$ splits as a graph of relatively hyperbolic groups satisfying the conditions of the Mj--Reeves combination theorem \cite{MR08} and an additional condition ---\emph{bounded girth hallways}, see \cref{defn:bounded_hallways}--- then a subgroup $H\leqslant G$ is relatively quasi-convex precisely if its intersections with the vertex groups are relatively quasi-convex and if it acts co-compactly on a subtree of the Bass--Serre tree.

In \cref{sec:graph_pairs} we define graph pairs and describe in detail the Feighn--Handel tightening procedure which was introduced in \cite{FH99} to describe presentations of finitely generated subgroups of mapping tori of free groups. Here we prove a slight strengthening of Feighn--Handel's main result, \cref{FHmain}, and derive some corollaries. In particular, we describe a useful HNN-extension decomposition $F*_{\phi}$ of a mapping torus of a free group $M(\psi)$ which will be part of the set-up in the later sections. Importantly, $F$ is finitely generated and $\phi$ identifies a free factor of $F$ with another subgroup.

In \cref{sec:peripherals}, we describe the collection $\mathcal{P}$ of subgroups of $M(\psi)$ and prove the first part of \cref{main}. The main idea is to analyse the action of $M(\psi)$ on the Bass--Serre tree associated with the HNN-extension decomposition $M(\psi)\cong F*_{\phi}$ from \cref{sec:graph_pairs}. A key property of this action is that it is \emph{relatively acylindrical}; that is, there is a constant $k$ so that any segment of length at least $k$ has stabiliser conjugate to an element in a free factor system of the free group $F$, see \cref{thm:acylindrical}.

In \cref{sec:rel_hyp_mapping_tori} we prove the second part of \cref{main}. Here we verify that all the conditions from the Mj--Reeves combination theorem \cite{MR08}, as well as our bounded girth hallways condition, are satisfied by the splitting $F*_{\phi}$.

In \cref{sec:locally_rel_qc} we complete the proof of \cref{main} and prove \cref{thm:one-relator}. The proof strategy for the local relative quasi-convexity statement will be to try and understand the induced splittings of finitely generated subgroups $H\leqslant M(\psi)$ with respect to the HNN-extension $F*_{\phi}$. Unfortunately, such induced splittings do not have finitely generated vertex and edge groups in general. However, by analysing direct limits of appropriately constructed graph pairs, we show that vertex groups of induced splittings are finitely generated relative to the adjacent edge groups. This will be enough for us to be able to deduce relative quasi-convexity of vertex stabilisers for the action of $H$ on the Bass--Serre tree for $F*_{\phi}$ and apply our relative quasi-convexity criterion, \cref{thm:qc_criterion}.

In \cref{sec:applications} we discuss some applications of \cref{main}.

\subsection*{Acknowledgements}
The author thanks Sam Hughes, Jean-Pierre Mutanguha and Henry Wilton for their comments on a previous version of this article. The author also thanks Mahan Mj for helpful discussions on the combination theorem for relatively hyperbolic groups. 

This work has received support from the grant 202450E223 (Impulso de líneas científicas estratégicas de ICMAT) and has benefitted from the hospitality of the Isaac Newton Institute for Mathematical Sciences where the last stages of this project were completed during the programme Operators, Graphs, and Groups.

\section{Preliminaries}
\label{sec:preliminaries}

\subsection{Graphs and graph maps}

A \emph{graph} for us will be a 1-dimensional CW-complex. We shall always assume that a cellular structure has been fixed on any given graph. A \emph{graph map} is a cellular map on graphs which sends 0-cells (\emph{vertices}) to 0-cells and open 1-cells (\emph{edges}) homeomorphically to open 1-cells. We will sometimes write $V(\Gamma)$ and $E(\Gamma)$ for the vertex and edge set of a graph $\Gamma$. Two edges are \emph{adjacent} if they both share an endpoint. A graph is \emph{pointed} if it comes with a basepoint, we shall usually denote this by a pair $(\Delta, u)$. A \emph{pointed graph map} $(\Gamma, v)\to (\Delta, u)$ is a graph map which sends the basepoint $v$ to the basepoint $u$. A rose graph is any graph with a single vertex. 

If $\alpha\colon I\to \Delta$ is a path, we write $o(\alpha)$ for the origin of $\alpha$ and $t(\alpha)$ for the target of $\alpha$. If $\alpha, \beta\colon I\to \Delta$ are two paths with $t(\alpha) = o(\beta)$, then we write $\alpha*\beta$ for their concatenation. When $\alpha\colon I\to \Delta$ is a loop at a vertex $u$, then we write $[\alpha]$ for the corresponding group element of $\pi_1(\Delta, u)$. Our graphs will be given the path metric induced by identifying each edge with a unit Euclidean interval. Then the \emph{length} $\len(\alpha)$ of a path $\alpha$ is its length with respect to this metric.

An \emph{immersion} of graphs is a locally injective graph map and is denoted by $\immerses$. Recall that if $\Gamma\immerses\Delta$ is an immersion of graphs, then the induced map on fundamental group(oid)s $\pi_1(\Gamma)\to \pi_1(\Delta)$ is injective. We will often use this fact without mention, identifying the fundamental group of $\Gamma$ with the image subgroup of $\pi_1(\Delta)$. The reader is directed towards Stallings article \cite{St83} for further details.

A graph $\Gamma$ is \emph{core} if it is the union of the images of all its immersed cycles $S^1\immerses \Gamma$. The \emph{core of a graph $\Gamma$} is the subgraph $\core(\Gamma)\subset\Gamma$ consisting of the union of all immersed cycles $S^1\immerses\Gamma$. A pointed graph $(\Gamma, v)$ is \emph{pointed core} if it is the union of the images of all its immersed loops $I\immerses \Gamma$  at the basepoint $v$. Similarly, the \emph{pointed core of $(\Gamma, v)$} is the pointed subgraph $\core(\Gamma, v)\subset (\Gamma, v)$ consisting of the union of the images of all immersed loops $I\immerses\Gamma$ at $v$.

If $(\Delta_1, u_1)$ and $(\Delta_2, u_2)$ are two pointed graphs, denote by $\Delta_1\vee \Delta_2$ the graph obtained from $\Delta_1\sqcup\Delta_2$ by identifying the two basepoints $u_1$ and $u_2$.

If $\Gamma$ is a graph and $\Lambda\subset \Gamma$ is a subgraph, recall that the relative Euler characteristic is:
\[
\chi(\Gamma, \Lambda) = \#\{\text{$0$-cells in $\Gamma - \Lambda$}\} - \#\{\text{$1$-cells in $\Gamma - \Lambda$}\}.
\]
Note that $\chi(\Gamma, \Lambda)$ is only defined if $\Gamma - \Lambda$ contains finitely many 0-cells and 1-cells. When $\Gamma$ is finite, $\chi(\Gamma, \Lambda) = \chi(\Gamma) - \chi(\Lambda)$.

\subsection{Folds and subgroups of free groups}
\label{sec:folds}

Let $\gamma\colon \Gamma\to \Delta$ be a graph map and suppose that $e_1, e_2$ are two edges with a common endpoint $v$ that both map to the same edge under $\gamma$. Then by identifying $e_1$ with $e_2$ (and by identifying the endpoints) via $\gamma$, we obtain a new graph $\Gamma'$, a graph map $\gamma'\colon \Gamma'\to \Delta$ and a quotient map $f\colon \Gamma\to \Gamma'$ such that $\gamma = \gamma'\circ f$. We say that $\Gamma'$ and $\gamma'$ are obtained from $\Gamma$ and $\gamma$ by a \emph{fold} or by \emph{folding $e_1$ and $e_2$}. Stallings showed that any graph map $\Gamma \to \Delta$ with $\Gamma$ a finite graph can be factored as a finite sequence of folds followed by a graph immersion \cite{St83}:
\[
\begin{tikzcd}
\Gamma = \Gamma_0 \arrow[r, "f_1"] & \Gamma_1  \arrow[r, "f_2"]  & \ldots  \arrow[r, "f_n"]  &\Gamma_n  \arrow[r, loop->] & \Delta
\end{tikzcd}
\]
The sequence of folds is not unique, but the final graph immersion $\Gamma_n\immerses\Delta$ is. The same holds true for infinite graphs after passing to a direct limit. This will be explained in \cref{sec:direct_limit} where we shall need it.

The following fact due to Stallings \cite{St83} will be very useful.

\begin{lemma}
\label{lem:subgroup_bijections}
Let $\Delta$ be a graph and let $u\in \Delta$ be a vertex. The maps given by the $\pi_1$-functor
\begin{align*}
\{(\Gamma, v) \immerses (\Delta, u)\mid (\Gamma, v) = \core(\Gamma, v) \} &\to \{ H \mid H\leqslant \pi_1(\Delta, u)\}\\
\{\Gamma \immerses \Delta \mid \Gamma = \core(\Gamma) \} &\to \{ [H] \mid 1\neq H\leqslant \pi_1(\Delta, u)\}
\end{align*}
are bijections, where here $[H]$ denotes the conjugacy class of $H$ in $\pi_1(\Delta, u)$.
\end{lemma}

The inverse of the maps from \cref{lem:subgroup_bijections} are given by taking the (pointed) core of the cover associated with the subgroup. If $H\leqslant \pi_1(\Delta, u)$ is a subgroup, the \emph{subgroup graph (immersion)} for $[H]$, which we shall denote by $\Gamma[H]\immerses\Delta$, is the unique immersion of a core graph such that the conjugacy class $[\pi_1(\Gamma[H], v)]$ in $\pi_1(\Delta, u)$ is precisely $[H]$. The \emph{pointed subgroup graph (immersion)} for $H$, which we shall denote by $(\Gamma(H), v)\immerses (\Delta, u)$, is the unique immersion of a pointed core graph such that $\pi_1(\Gamma(H), v)$ is precisely $H$. If $\delta\in \pi_1(\Delta, u)$, then we will abuse notation and write $\Gamma(\delta)$ when we mean $\Gamma(\langle \delta\rangle)$.

\begin{lemma}
\label{lem:factorise}
Let $\gamma\colon(\Gamma, v)\immerses(\Delta, u)$ and $\lambda\colon (\Lambda, w)\immerses(\Delta, u)$ be immersions of graphs with $(\Gamma, v)$ pointed core. If $\gamma_*\pi_1(\Gamma, v)$ is contained in $\lambda_*\pi_1(\Lambda, w)$, then $\gamma$ factorises uniquely through $\lambda$. If $\gamma_*\pi_1(\Gamma, v)$ is conjugate into $\lambda_*\pi_1(\Lambda, w)$, then the restriction of $\gamma$ to $\core(\Gamma)$ factorises through $\lambda$.
\end{lemma}

If $\gamma, \lambda\colon \Gamma, \Lambda\immerses\Delta$ are two graph immersions, their pullback, which exists and can be described explicitly (see \cite{St83}), is denoted by $\Gamma\times_{\Delta}\Lambda$. The pullback comes with natural projections maps $p_{\Gamma}, p_{\Lambda}\colon \Gamma\times_{\Delta}\Lambda\to \Gamma, \Lambda$.

The following is explained in \cite{St83}.

\begin{lemma}
\label{lem:double_coset}
Let $\gamma, \lambda\colon(\Gamma, v), (\Lambda, w)\immerses(\Delta, u)$ be immersions of graphs and let $\Gamma\times_{\Delta}\Lambda$ be their pullback. There is a bijection
\[
\pi_0(\core(\Gamma\times_{\Delta}\Lambda)) \to \{ \pi_1(\Gamma, v)\cdot g\cdot \pi_1(\Lambda, w) \mid \pi_1(\Gamma, v)^g\cap \pi_1(\Lambda, w) \neq 1\}
\]
given by choosing a vertex $x\in \Theta\in \pi_0(\core(\Gamma\times_{\Delta}\Lambda))$ and choosing any pair of paths $\alpha\colon I\to \Gamma$ and $\beta\colon I\to \Lambda$ connecting $v$ with $p_{\Gamma}(x)$ and $w$ with $p_{\Lambda}(x)$ respectively, and sending
\[
\Theta \mapsto \pi_1(\Gamma, v)\cdot [\gamma\circ\alpha*\overline{\lambda\circ\beta}]\cdot \pi_1(\Lambda, w).
\]
Explicitly, we have $\pi_1(\Theta, x)^{[\overline{\lambda\circ\beta}]} = \pi_1(\Gamma, v)^{[\gamma\circ\alpha*\overline{\lambda\circ\beta}]}\cap \pi_1(\Lambda, w)$.
\end{lemma}

\subsection{Free factor systems}

If $F$ is a free group, a collection of subgroups $\{A_{\alpha}\}$ of $F$ is a \emph{free factor system} if for each $\alpha$ there is some element $f_{\alpha}\in F$ such that $\Asterisk_{\alpha}A_{\alpha}^{f_{\alpha}}$ is a free factor of $F$.

We record the following well-known fact which can be seen directly from \cref{lem:double_coset}.

\begin{lemma}
\label{ff_malnormal}
If $F$ is a free group and $\{A_{\alpha}\}$ is a free factor system of $F$, then $\{A_{\alpha}\}$ forms a malnormal collection.
\end{lemma}

Free factor systems behave well when intersecting with subgroups.

\begin{lemma}
\label{ff_system}
Let $F$ be a free group, let $\{A_{\alpha}\}$ be a free factor system for $F$ and let $H\leqslant F$ be a subgroup. For each $\alpha$, let $\{f_{\alpha, \beta}\}$ be any collection of elements in distinct $A_{\alpha}, H$ double cosets such that $A^{f_{\alpha, \beta}}\cap H\neq 1$, then $\{A_{\alpha}^{f_{\alpha, \beta}}\cap H\}$ is a free factor system for $H$.

In particular, if each $A_{\alpha}$ is contained in $H$, then $\{A_{\alpha}\}$ is a free factor system for $H$. 
\end{lemma}

\begin{proof}
Let $R$ be a rose graph such that $\pi_1(R) = F$ and, since $\{A_{\alpha}\}$ is a free factor system, we may assume that for each $\alpha$ there is a subgraph $\Lambda_{\alpha}\subset R$ such that $\pi_1(\Lambda_{\alpha})$ is conjugate to $A_{\alpha}$ and such that all the $\Lambda_{\alpha}$ pairwise intersect each other at the unique vertex. Let $\Gamma = \Gamma(H)$, $\Lambda = \cup_{\alpha}\Lambda_{\alpha}$ and consider the pullback $\Gamma\times_R\Lambda$. Since $\Lambda$ is a subgraph of $R$, the projection map $p_{\Gamma}\colon \Gamma\times_R\Lambda\to \Gamma$ is an embedding. Now \cref{lem:double_coset} implies the result.
\end{proof}

A free factor system $\{B_{\beta}\}$ of $F$ \emph{refines} a free factor system $\{A_{\alpha}\}$ if for each $\beta$ there is an $\alpha$ such that $B_{\beta}$ is conjugate into $A_{\alpha}$. It \emph{properly refines} $\{A_{\alpha}\}$ if it refines $\{A_{\alpha}\}$ and if some $B_{\beta}$ is conjugate to a proper free factor of some $A_{\alpha}$ or if there is some $A_{\alpha}$ so that no $B_{\beta}$ is conjugate into $A_{\alpha}$.

If $F$ is a free group and $A\leqslant F$ is a finitely generated subgroup, the \emph{reduced rank} of $A$ is $\rr(A) = \max{\{\rk(A) - 1, 0\}}$.

\begin{lemma}
\label{ff_inequality}
Let $F$ be a free group, let $\{A_{\alpha}\}$ be a free factor system consisting of finitely many finitely generated free factors. If $\{B_{\alpha}\}$ is a free factor system refining $\{A_{\alpha}\}$, then 
\[
\sum_{\beta}\rr(B_{\beta}) \leqslant \sum_{\alpha}\rr(A_{\alpha})
\]
with equality if and only if $\{B_{\beta}\}$ does not properly refine $\{A_{\alpha}\}$.
\end{lemma}

\begin{proof}
The claimed inequality holds by Grushko's theorem. The inequality certainly becomes an equality when $\{B_{\beta}\}$ does not properly refine $\{A_{\alpha}\}_{\alpha}$. Now suppose that the inequality is an equality. We have $\sum_{\beta}\rr(B_{\beta}) = \sum_{\beta}\rk(B_{\beta}) - \#\{\beta\}$ and $\sum_{\alpha}\rr(A_{\alpha}) = \sum_{\alpha}\rk(A_{\alpha}) - \#\{\alpha\}$. Hence,  $\sum_{\beta}\rk(B_{\beta}) = \sum_{\alpha}\rk(A_{\alpha})$ and $\#\{\beta\} = \#\{\alpha\}$ which implies that $\{B_{\beta}\}$ does not properly refine $\{A_{\alpha}\}$.
\end{proof}

\subsection{Relatively hyperbolic spaces}

The reader is invited to consult Bridson--Haefliger \cite{BH99} for the relevant background on hyperbolic spaces and groups.

Let $X$ be a space and let $\mathcal{H}$ be a collection of closed subspaces. The \emph{coned-off space} $\widehat{X}$ corresponding to the pair $(X, \mathcal{H})$ is the space obtained from $X$ by adding a point $v_{\alpha}$ for each $H_{\alpha}\in \mathcal{H}$ and connecting each point in $H_{\alpha}$ to $v_{\alpha}$ by an interval of length $\frac{1}{2}$. If $X$ is a geodesic path metric space, then so is $\wh{X}$.

Let $X$ be a geodesic path metric space. If $K\geqslant 0$, we say a subspace $H\subset X$ is \emph{$K$-quasi-convex} if every geodesic in $X$ connecting two points in $H$ remains at distance at most $K$ from $H$. It is \emph{quasi-convex} if it is $K$-quasi-convex for some $K$.

Let $\mathcal{H}$ be a collection of closed $K$-quasi-convex subspaces of $X$. The collection $\mathcal{H}$ is said to be \emph{uniformly separated} if there exists an $\epsilon>0$ such that for each pair $H_{\alpha}, H_{\beta}\in \mathcal{H}$, we have $d(H_{\alpha}, H_{\beta})\geqslant \epsilon$. If $D\geqslant 0$, the collection $\mathcal{H}$ is said to be \emph{mutually $D$-cobounded} if any nearest point projection of any $H_{\alpha}\in \mathcal{H}$ to any other $H_{\beta}\in \mathcal{H}$ has diameter at most $D$. It is \emph{mutually cobounded} if it is mutually $D$-cobounded for some $D$.

\begin{definition}
If $X$ is a geodesic path metric space and $\mathcal{H}$ is a collection of quasi-convex, uniformly separated and mutually disjoint closed subspaces of $X$, then the pair $(X, \mathcal{H})$ is said to be \emph{weakly relatively hyperbolic} if the coned-off space $\widehat{X}$ is hyperbolic (in the sense of Gromov). The pair $(X, \mathcal{H})$ is \emph{relatively hyperbolic} if it is weakly hyperbolic and if $\mathcal{H}$ satisfies bounded penetration (see \cite[Definition 2.8]{MR08}).
\end{definition}

\begin{remark}
\label{rem:mutually_cobounded}
We do not define bounded penetration here as it is a technical definition which we shall not use. We only remark that if $X$ is hyperbolic and $\mathcal{H}$ is mutually cobounded, then $(X, \mathcal{H})$ satisfies bounded penetration by \cite[Lemma 2.7]{MR08} and so is relatively hyperbolic.
\end{remark}

\subsection{Relatively hyperbolic groups and quasi-convex subgroups}
\label{sec:rel_hyp}

A \emph{group pair} is a pair $(G, \mathcal{P})$ where $G$ is a group and $\mathcal{P}$ is a collection of subgroups of $G$. A subset $S\subset G$ is a \emph{generating set for $(G, \mathcal{P})$} if $G = \langle S, \bigcup_{P\in \mathcal{P}}P\rangle$. If $S\subset G$ is a generating set for $(G, \mathcal{P})$, then define the \emph{coned-off Cayley graph} $\wh{\Gamma}(G, \mathcal{P}, S)$ to be the coning-off of the pair $(\Gamma, \mathcal{H})$ where $\Gamma$ is the Cayley graph of $G$ (with respect to $S$) and where $\mathcal{H} = \{gP\}_{g\in G, P\in\mathcal{P}}$.

\begin{definition}
\label{def:rel_hyp}
A group pair $(G, \mathcal{P})$, with $\mathcal{P}$ a finite collection of subgroups, is \emph{relatively hyperbolic} if for some (any) generating set $S\subset G$ for $(G, \mathcal{P})$, the coned-off Cayley graph $\wh{\Gamma} = \wh{\Gamma}(G, \mathcal{P}, S)$ is hyperbolic and if $\mathcal{P}$ has Bounded coset penetration in $\hat{\Gamma}$ (see \cite[Definition 3.6]{Hr10}).
\end{definition}

\begin{definition}
If $(G, \mathcal{P})$ is a relatively hyperbolic group pair, a subgroup $A\leqslant G$ is \emph{relatively quasi-convex} if for some (any) finite generating set $S\subset G$ for $(G, \mathcal{P})$, the subset $A\subset \wh{\Gamma}(G, \mathcal{P}, S)$ is quasi-convex.
\end{definition}

Note that a relatively quasi-convex subgroup does not have to be finitely generated. However, a relatively quasi-convex subgroup is always undistorted and relatively hyperbolic, see \cite{Hr10}.

An \emph{action} of a group $G$ on a pair $(X, \mathcal{H})$ is an action of $G$ on $X$ preserving $\mathcal{H}$. An action of $G$ on a pair $(X, \mathcal{H})$ extends naturally to an action of $G$ on $\widehat{X}$. If $G$ acts on a pair $(X, \mathcal{H})$, there is an \emph{associated group pair} $(G, \mathcal{P})$ where
\[
\mathcal{P} = \{\stab_G(v_{\alpha})\}_{G\cdot H_{\alpha}\in G\backslash\mathcal{H}}
\]
are stabilisers of $G$-orbit representatives of the subspaces in $\mathcal{H}$. 

If $(G, \mathcal{P})$ is relatively hyperbolic, then it is clear that it is the group pair associated to the natural action of $(G, \mathcal{P})$ on its coned-off Cayley graph $\wh{\Gamma}(G, \mathcal{P}, S)$ (for any finite generating set $S\subset G$ of $(G, \mathcal{P})$). Conversely, we have the following lemma which follows by noting that the coned-off space $\widehat{X}$ is quasi-isometric to a coned-off Cayley graph for $(G, \mathcal{P})$ with respect to a finite generating set for $G$.

\begin{lemma}
\label{lem:action_def}
Let $G$ be a finitely generated group acting geometrically on a relatively hyperbolic space pair $(X, \mathcal{H})$. If $(G, \mathcal{P})$ is the associated group pair, then $(G, \mathcal{P})$ is relatively hyperbolic. Furthermore, if $A\leqslant G$ is a subgroup, then $A$ is quasi-convex if some (any) co-compact $A$-invariant subspace of $\wh{X}$ is quasi-convex.
\end{lemma}

For this article, the most important examples of relatively hyperbolic spaces and quasi-convex subsets come from graphs. The following lemma will be used when we wish to apply the combination theorem and our quasi-convexity criterion to mapping tori of graphs.

\begin{lemma}
\label{lem:tree_rel_hyp}
Let $\Gamma$ be a finite graph and let $\{\Delta_{\alpha}\}$ be a collection of disjoint connected subgraphs of $\Gamma$. If $\widetilde{\Gamma}$ is the universal cover of $\Gamma$ and $\mathcal{H}$ denotes the union of all preimages of each $\Delta_{\alpha}$ in $\widetilde{\Gamma}$, then 
\begin{itemize}
\item $\left(\widetilde{\Gamma}, \mathcal{H}\right)$ is relatively hyperbolic.
\item if $\lambda\colon\Lambda\to \Gamma$ is an immersion with $\Lambda$ core, then $\lambda_*(\pi_1(\Lambda))$ is relatively quasi-convex if and only if $\Lambda - \bigcup_{\alpha}\lambda^{-1}(\Delta_{\alpha})$ consists of finitely many 0-cells and 1-cells.
\end{itemize}
\end{lemma}

\begin{proof}
Since $\{\Delta_{\alpha}\}$ is a collection of disjoint connected subgraphs of $\Gamma$, their preimages in $\wt{\Gamma}$ form a disjoint collection of subtrees which are thus mutually cobounded. Since $\Gamma$ is a finite graph, they are also uniformly separated and quasi-convex. Thus, $(\widetilde{\Gamma}, \mathcal{H})$ is relatively hyperbolic by \cref{rem:mutually_cobounded}.

Choose some basepoint in $\Lambda$ and let $x\in \widetilde{\Gamma}$ be a lift. The $H$-orbit, $H\cdot x$ is $K$-quasi-convex in the coning off of $(\widetilde{\Gamma}, \mathcal{H})$ if and only if each vertex in the coning off of $(\Lambda, \bigcup_{\alpha}\lambda^{-1}(\Delta_{\alpha}))$ is at distance at most $K$ from the basepoint. Since $\Gamma$ is finite, $\Lambda$ is locally finite and so (using also the fact that $\Lambda$ is core) such a $K$ exists if and only if $\Lambda - \bigcup_{\alpha}\lambda^{-1}(\Delta_{\alpha})$ consists of finitely many 0-cells and 1-cells. Using \cref{lem:action_def} completes the proof.
\end{proof}

\section{Criteria for relative hyperbolicity and quasi-convexity}
\label{sec:criteria}

In this section we present the Mj--Reeves (strong) combination theorem \cite{MR08} and prove a criterion for relative quasi-convexity of subgroups which may be of independent interest. The main references for this section are Hruska \cite{Hr10} and Mj--Reeves \cite{MR08}.

\subsection{The combination theorem for trees of relatively hyperbolic spaces}

A \emph{graph of spaces} for us will be a space $X$ together with data 
\[
\left(\Gamma, \{X_v\}_{v\in V(\Gamma)}, \{X_e\}_{e\in E(\Gamma)}, \{\partial^{\pm}_e\}_{e\in E(\Gamma)}\right)
\]
where:
\begin{enumerate}
\item $\Gamma$ is a graph called the \emph{underlying graph}.
\item For each vertex $v\in V(\Gamma)$, there is an associated connected \emph{vertex space} $X_v\subset X$.
\item For each edge $e\in E(\Gamma)$, there is an associated connected \emph{edge space} $X_e\subset X$.
\item For each edge $e\in E(\Gamma)$, if $e^-, e^+$ are the origin and target of $e$ respectively, there are maps $\partial_e^{\pm}\colon X_e\to X_{e^{\pm}}$ which are injective on $\pi_1$.
\item The space $X$ is
\[
X = \frac{\left(\bigsqcup_{v\in V(\Gamma)}X_v\right)\sqcup\left(\bigsqcup_{e\in E(\Gamma)}[-1, 1]\times X_e\right)}{\{\partial^{\pm}_e(x) \sim (\pm1, x) \mid \forall e\in E(\Gamma), \, \forall x\in X_e\}}
\]
and comes with a natural projection map $\pi\colon X \to \Gamma$.
\end{enumerate}
A tree of spaces is a graph of spaces with underlying graph $\Gamma$ a tree. Note that the universal cover of a graph of spaces naturally has the structure of a tree of spaces where the underlying tree is the Bass--Serre tree for the corresponding graph of groups (given by the $\pi_1$-functor) and where each vertex and edge space is the universal cover of a vertex space and edge space.

\begin{definition}[Tree of (relatively) hyperbolic spaces]
A tree of metric spaces $X$ (with underlying tree $T$) is a \emph{tree of relatively hyperbolic spaces} if there exist a constant $\delta\geqslant 0$ such that the following holds:
\begin{enumerate}
\item $X$ is a metric space and the metrics on the vertex spaces $(X_v, d_v)$ and the edge spaces $(X_e, d_e)$ coincide with the induced path metrics.
\item Each vertex space $\pi^{-1}(v) = X_v$ is relatively hyperbolic with respect to a collection of subspaces $\mathcal{H}_v$ and the coned-off space $\wh{X}_v$ is $\delta$-hyperbolic. Moreover, the inclusions $\iota_v\colon X_v\injects X$ are required to be uniformly proper.
\item Each edge space $\pi^{-1}(m_e) = X_e$ (here $m_e\in e$ is the midpoint of the edge $e\subset T$) is relatively hyperbolic with respect to a collection of subspaces $\mathcal{H}_e$ and the coned-off space $\wh{X}_e$ is $\delta$-hyperbolic.
\item For each edge $e$, the inclusion $X_e\times(-1, 1)\injects X$ is an isometry onto its image.
\end{enumerate}
If the collections $\mathcal{H}_v, \mathcal{H}_e$ are all empty, then $X$ is a \emph{tree of hyperbolic spaces}.
\end{definition}

\begin{definition}[Strictly type preserving]
A tree of relatively hyperbolic spaces $X$ satisfies the \emph{strictly type preserving condition} if for each edge $e\subset T$, we have
\begin{enumerate}
\item For each $H_{\alpha}\in \mathcal{H}_{e^{\pm}}$, we have $(\partial_e^{\pm})^{-1}(H_{\alpha})\subset H_{\beta}\in\mathcal{H}_e$ for some $\beta$.
\item For each $H_{\alpha}\in \mathcal{H}_e$, we have $\partial^{\pm}_e(H_{\alpha})\subset H_{\beta}\in \mathcal{H}_{e^\pm}$ for some $\beta$.
\end{enumerate}
\end{definition}

The strictly type preserving condition allows us to define the \emph{induced tree of coned-off spaces} via the induced maps $\wh{\partial}_{e}^{\pm}\colon \wh{X}_e\to \wh{X}_{e^{\pm}}$. We shall denote this space by $\overline{X}$. Note that the underlying tree for $\overline{X}$ is the same as that of $X$. 

The \emph{cone locus} of $\overline{X}$ is the forest with underlying vertex set the cone points in the vertex spaces $\wh{X}_v\subset\overline{X}$ and with edge set the products of cone points in the edge spaces with $[-1, 1]$. The components of the cone locus are the \emph{maximal cone-subtrees} $S\in \mathcal{S}$.

For each maximal cone-subtree $S$ of $\overline{X}$, one can form the \emph{maximal cone-subtree of horosphere-like spaces} $C$ defined as the tree of spaces with underlying tree $S$ and vertex and edge spaces the subsets $H_{\alpha}\in \mathcal{H}_v$, $H_{\beta}\in \mathcal{H}_e$ corresponding to the vertices and edges in $S$. Denote the collection of these spaces by $\mathcal{C}$. We will denote by $\wh{X}$ the coned-off space for the pair $(X, \mathcal{C})$. Note that $\wh{X}$ is obtained from $\overline{X}$ by collapsing the maximal cone-subtrees to points.

\begin{definition}[Qi-embedded]
A tree of (relatively) hyperbolic spaces $X$ satisfies the \emph{quasi-isometrically (qi)-embedded condition} if there are constants $K, C\geqslant 0$ such that for each edge space $X_e$, the maps $\partial_e^{\pm}\colon X_e\to X_{e^{\pm}}$ are $(K, C)$-quasi-isometric embeddings.
\end{definition}

\begin{definition}[Qi-preserving electrocution]
A tree of relatively hyperbolic spaces $X$ that is strictly type preserving satisfies the \emph{qi-preserving electrocution condition} if there are constants $K, C\geqslant 0$ such that for each edge space $X_e$, the induced maps $\wh{\partial}_e^{\pm}\colon \wh{X}_e\to \wh{X}_{e^{\pm}}$ are $(K, C)$-quasi-isometric embeddings.
\end{definition}

Let $X$ be a tree of geodesic path metric spaces with underlying tree $T$. Following \cite{BF92}, a \emph{hallway of length $2m$} is a map $h\colon [-m, m]\times I\to X$ such that
\begin{enumerate}
\item\label{itm:cond_1} $h^{-1}(\{0\}\times X_e)\subset \{-m, \ldots, m\}\times I$ for each edge $e\subset T$.
\item $h$ is transverse, relative to condition (\ref{itm:cond_1}), to each edge space $X_e$.
\item $h\mid\{i\}\times I$ is a geodesic in the corresponding edge space.
\end{enumerate}
The \emph{girth} of a hallway $h$ is the length of the path $h\mid\{0\}\times I$. A hallway $h$ is \emph{essential} if its projection to $T$ is a path without backtracking. It is \emph{$\rho$-thin} if $d(h(i, t), h(i+1, t))\leqslant \rho$ for all $i, t$. It is \emph{$\lambda$-hyperbolic} if 
\[
\lambda\cdot \len(h\mid \{0\}\times I) \leqslant \max{\{\len(h\mid \{-m\}\times I), \len(h\mid\{m\}\times I)\}}
\]
where we recall that $\len(-)$ denotes the length of a path. If $X$ is an induced tree of coned-off spaces, then $h$ is \emph{cone-bounded} if $h\mid [-m, m]\times\{0\}$ and $h\mid[-m, m]\times\{1\}$ lie in the cone locus.

\begin{remark}
The definition of a cone-bounded hallway presented here is slightly different to that presented in \cite{MR08}; there, a cone-bounded hallway is only required to have $h\mid\{i\}\times\partial I$ lying in the cone locus for each $i\in \{-m, \ldots, m\}$. However, the definition of cone-bounded hallway that is used in the proof of their main theorem (\cref{thm:combination} below) is the one that we have given, see the proof of \cite[Proposition 4.4]{MR08}.
\end{remark}

\begin{definition}[Hallways flare]
A tree of geodesic path metric spaces $X$ is said to satisfy the \emph{hallways flare condition} if there exist $\lambda>1$, $m\geqslant 1$ such that for all $\rho\geqslant 0$, there is a constant $H(\rho)$ such that every essential $\rho$-thin hallway of length $2m$ and of girth at least $H(\rho)$ is $\lambda$-hyperbolic.
\end{definition}

\begin{definition}[Cone-bounded hallways strictly flare]
A tree of coned-off spaces $\overline{X}$ (associated to a tree of relatively hyperbolic spaces $X$) is said to satisfy the \emph{cone-bounded hallways strictly flare condition} if there exist $\lambda>1$, $m\geqslant 1$ such that every cone-bounded essential hallway in $\overline{X}$ of length $2m$ is $\lambda$-hyperbolic.
\end{definition}

Below we state the combination theorem for relatively hyperbolic groups due to Mj--Reeves \cite{MR08}. See also the work of Gautero \cite{Ga16} for an alternative proof. 

\begin{theorem}
\label{thm:combination}
Let $X$ be a tree of relatively hyperbolic spaces such that:
\begin{enumerate}
\item $X$ satisfies the strictly type preserving, the qi-embedded and the qi-preserving electrocution condition.
\item The induced tree of coned-off spaces $\overline{X}$ satisfies the hallways flare and the cone-bounded hallways strictly flare condition.
\end{enumerate}
Then $(X, \mathcal{C})$ is relatively hyperbolic.
\end{theorem}

\subsection{Relatively quasi-convex vertex spaces}

We now investigate (relatively) quasi-convex subspaces of a tree of (relatively) hyperbolic spaces. The following definition is a special case of \cite[Definition 4.26]{Mj20}

\begin{definition}[Bounded girth hallways]
\label{defn:bounded_hallways}
A tree of geodesic metric spaces $X$ is said to satisfy the \emph{bounded girth hallways condition} if there is some $m\geqslant 1$ such that for all $\rho\geqslant 0$, there is a constant $G(\rho)$ such that every essential $\rho$-thin hallway of length $2m$ has girth at most $G(\rho)$.
\end{definition}

Note that if an induced tree of spaces $X$ satisfies the bounded girth hallways condition, then it satisfies the hallways flare condition.

The following is \cite[Corollary 4.3]{Mi04}, but can also be seen by \cite[Proposition 4.27]{Mj20}.

\begin{theorem}
\label{thm:bounded_hallways}
Let $X$ be a tree of hyperbolic spaces satisfying the qi-embedded and the bounded girth hallways condition. Then $X$ is hyperbolic and there is a constant $K\geqslant 0$ so that each vertex space of $X$ is $K$-quasi-convex in $X$.
\end{theorem}

Combining \cref{thm:combination} with \cref{thm:bounded_hallways}, we may obtain a criterion for when vertex spaces are quasi-convex in the relatively hyperbolic setting. See also \cite{KS24} and \cite{To25} for related criteria.

\begin{corollary}
\label{cor:bounded_hallways}
Let $X$ be a tree of relatively hyperbolic spaces such that:
\begin{enumerate}
\item $X$ satisfies the strictly type preserving, the qi-embedded and the qi-preserving electrocution condition.
\item The induced tree of coned-off spaces $\overline{X}$ satisfies the bounded girth hallways condition and the cone-bounded hallways strictly flare condition.
\end{enumerate}
Then $(X, \mathcal{C})$ is relatively hyperbolic and there is a constant $K\geqslant 0$ so that each vertex space $\wh{X}_v$ is $K$-quasi-convex in $\wh{X}$.
\end{corollary}

\begin{proof}
The fact that $(X, \mathcal{C})$ is relatively hyperbolic is \cref{thm:combination}. Then \cref{thm:bounded_hallways} implies that there is a constant $K$ so that each vertex space $\wh{X}_v$ is $K$-quasi-convex in the induced tree of coned-off spaces $\overline{X}$. Letting $\wh{\overline{X}}$ be the coning off of $\overline{X}$ with respect to the maximal cone-subtrees, \cite[Lemma 2.4]{MR08} states that there is a constant $K'$ such that any geodesic in $\wh{\overline{X}}$ lies in the $K'$-neighbourhood of any geodesic in $\overline{X}$ connecting the same pair of endpoints. Thus, since each geodesic in $\overline{X}$ connecting two points in $\wh{X}_v$ lies in the $K$-neighbourhood of $\wh{X}_v$, it follows that each geodesic in $\wh{\overline{X}}$ connecting two points in $\wh{X}_v$ lies in the $(K+K')$-neighbourhood of $\wh{X}_v$. Hence, $\wh{X}_v$ is $(K+K')$-quasi-convex in $\wh{\overline{X}}$ for each vertex $v$. Since the map $\wh{\overline{X}}\to\wh{X}$ which collapses all the cones to points is a quasi-isometry, the result follows.
\end{proof}

\subsection{Quasi-geodesics in a tree of hyperbolic spaces}
\label{sec:quasi-geodesics}

In this section we shall analyse quasi-geodesics in trees of hyperbolic spaces which satisfy the bounded girth hallways condition. The aim will be to apply this analysis to obtain a criterion for when a subgroup of a graph of relatively hyperbolic groups is quasi-convex.

We begin with a lemma which appears in \cite{BF92} and is attributed to Gromov \cite{Gr87}. We only state the special case we need.

\begin{lemma}
\label{lem:qg_4-gon}
Let $X$ be a $\delta$-hyperbolic metric space and let $K>0, C\geqslant 0$ be constants. There is a constant $K'$ such that the following holds.  If $\Delta\colon S^1\to X$ is a $(K, C)$-quasi-geodesic 4-gon, then there is a metric tree $S$ with 4 vertices of degree one and a map (resolution) $r\colon D^2\to S$ so that the following holds:
\begin{enumerate}
\item For $a, b\in S^1 = \partial D^2$, $d_X(\Delta(a), \Delta(b))\leqslant d_S(r(a), r(b)) + K'$.
\item For each $s\in S$, $r^{-1}(s)\subset D^2$ is a properly embedded finite tree.
\item For each open edge $e\subset S$, $r^{-1}(e) \cong e\times I$.
\end{enumerate}
\end{lemma}

If $M\geqslant 0$ and $Y\subset X$, denote by $N_M(Y)$ the $M$-neighbourhood of $Y$ in $X$.

\begin{proposition}
\label{prop:qg_track}
Let $X$ be a tree of hyperbolic spaces satisfying the qi-embedded and the bounded girth hallways condition and let $K>0, C, L\geqslant 0$ be constants. There is a constant $M\geqslant 0$ such that the following holds. 

Let $\gamma, \lambda\colon I\to X$ be two paths that both project to the same path without backtracking in $T$ and such that $d(o(\gamma), o(\lambda)), d(t(\gamma), t(\lambda))\leqslant L$. If $\gamma$ is a $(K, C)$-quasi-geodesic, and each subpath of $\lambda$ in each vertex space $X_v$ is a $(K, C)$-quasi-geodesic, then $\gamma\subset N_M(\lambda)$.
\end{proposition}

\begin{proof}
Let $m\geqslant 1$ be the constant from \cref{defn:bounded_hallways} guaranteed to exist by assumption. The tree of hyperbolic spaces $X$ is thus hyperbolic by the Bestvina--Feighn combination theorem \cite{BF92}.

Since each edge space uniformly quasi-isometrically embeds in the adjacent vertex spaces, after possibly performing a path homotopy to $\gamma, \lambda$ and increasing $K$ and $C$ (by a constant not depending on $\gamma, \lambda$), we may write
\begin{align*}
\gamma &= \gamma_0*e_1*\gamma_1*\ldots*e_n*\gamma_n\\
\lambda &= \lambda_0*f_1*\lambda_1*\ldots*f_n*\lambda_n
\end{align*}
where each $\gamma_i, \lambda_i$ are $(K, C)$-quasi-geodesics in the same vertex space $X_{v_i}\subset X$ and each $e_i, f_i$ are paths of length one that both project to the same edge in $T$.

Let $p_0$ be a geodesic (in its vertex space) connecting $o(\gamma_0)$ with $o(\lambda_0)$ and let $q_n$ be a geodesic (in its vertex space) connecting $t(\gamma_n)$ with $t(\lambda_n)$. For each $i>0$, let $p_i$ be a path connecting $o(\gamma_i)$ with $o(\lambda_i)$ so that $p_i$ is the image of a geodesic in the edge space $X_{\pi(e_i)}$ under $\partial_{\pi(e_i)}^+$. For each $i<n$, let $q_i$ be the image of $p_{i+1}$ under the edge map. In particular, each point along $q_i$ has a corresponding point (at distance 1) on $p_{i+1}$. Since the edge maps are all quasi-isometric embeddings, each $p_i$, $q_i$ are quasi-geodesics in their vertex spaces, with the constants depending only on the data defining $X$. For each $i$,
\[
\gamma_i*q_i*\overline{\lambda}_i*\overline{p}_i
\]
is a quasi-geodesic 4-gon in $X_{v_i}$. For each $i$, let $S_i$ be the metric tree and $\Delta_i$, $r_i$ the maps from \cref{lem:qg_4-gon} for the $i^{\text{th}}$ quasi-geodesic 4-gon. Let $K'$ be the constant from \cref{lem:qg_4-gon}.

Now let $x$ be an arbitrary point on $\gamma_i*q_i*\overline{\lambda}_i*\overline{p}_i$. Consider the point $r_i(x)\in S_i$. Then in each of the sides of $\gamma_i*q_i*\overline{\lambda}_i*\overline{p}_i$ that do not contain $x$, there is at most one point $y$ (and at least one overall) such that $r_i(y) = r_i(x)$. If there is such a point $y\in p_i$, then there is a corresponding point $y'\in q_{i-1}$ at distance $1$ from $y$. If there is such a point in $y\in q_i$, then there is a corresponding point $y'\in p_{i+1}$ at distance $1$ from $y$. By considering $r_{i-1}(y')$ in the first case or $r_{i+1}(y')$ in the second case and continuing in this way, we may find a sequence of points $x = x_0, x_1, \ldots, x_{2k}, x_{2k+1}$ such that the following holds:
\begin{enumerate}
\item For each $j\neq 0, k$, there is an $i$ so that $x_{2j}\in p_{i+j}$ and $x_{2j+1}\in q_{i+j}$ (or $q_{i-j}$ and $p_{i-j}$ if the sequence goes in the opposite direction) and $d(x_{2j}, x_{2j+1})\leqslant K'$.
\item For each $j\neq 0, k$, the pair $x_{2j-1}, x_{2j}$ are adjacent points in $q_{i+j}, p_{i +j+1}$ respectively (or $p_{i-j}, q_{i-j-1}$ respectively).
\item $x_{2k+1}$ lies on $\gamma*q_n*\overline{\lambda}*\overline{p}_0$.
\end{enumerate}
Call the sequence of points a \emph{crossing sequence} for $x$. For each $k$ there is a path of length at most $K'+1$ connecting $x_{2j}$ with $x_{2(j+1)}$. By concatenating all these paths together, we may obtain a path $\alpha_x$ connecting $x = x_0$ with $x_{2k}\in \gamma*q_n*\overline{\lambda}*\overline{p}_0$ of length at most $k(K'+1)$. Call the path $\alpha_x$ a crossing path for $x$. If $x\in \gamma*q_n*\overline{\lambda}*\overline{p}_0$, then call $\alpha_x$ a transverse crossing path. Note that every point in each $\gamma_i*q_i*\overline{\lambda}_i*\overline{p}_i$ lies in at least one transverse crossing path.

\begin{claim}
\label{claim:qg}
If $\alpha$ is a crossing path, then $\alpha$ is a $(K'+1, 2K')$-quasi-geodesic.
\end{claim}

\begin{proof}
This can be seen by considering the projection of $\alpha$ to $T$.
\end{proof}

All constants in the rest of the proof will depend only on the hyperbolicity constants for the vertex spaces $X_v$ and the constants $K$, $C$, $K'$, $L$, $m$ and $G(K'+1)$.

\begin{claim}
\label{claim:short_path}
There is a constant $K''$ such that the following holds. If $y\in p_i$ is a point along $p_i$ so that there is a transverse crossing path $\alpha$ which traverses $y$, begins and ends on $\gamma$ and so that $\pi(\alpha)$ is a segment of length at most $2m$, then $d(o(p_i), y)\leqslant K''$.
\end{claim}

\begin{proof}
Note that $o(p_i)$ lies in between $o(\alpha)$ and $t(\alpha)$ along $\gamma$. Since $\alpha$ is a $(K'+1, 2K')$-quasi-geodesic by \cref{claim:qg} and $\gamma$ is a $(K, C)$-quasi-geodesic, there is a constant $\kappa$, depending only on $K, C, K'$, so that $\alpha\subset N_{\kappa}(\gamma[o(\alpha), t(\alpha)])$ and $\gamma[o(\alpha), t(\alpha)]\subset N_{\kappa}(\alpha)$. In particular, $o(p_i)$ lies at distance at most $\kappa$ from some point on $\alpha$. Since $\alpha$ has length at most $2m(K'+1)$, the claim follows by setting $K'' = 2m(K'+1) + \kappa + 1$.
\end{proof}

\begin{claim}
\label{claim:thin_hallway}
There is a constant $K'''$ such that the following holds. If $y, z\in p_i$ are points so that there are transverse crossing paths traversing $y$ and $z$ which project to segments of length at least $2m$ in $T$, then $d(y, z)\leqslant K'''$.
\end{claim}

\begin{proof}
Let $\alpha_1, \alpha_2$ be transverse crossing paths traversing $y, z$ respectively so that $\pi(\alpha_1), \pi(\alpha_2)$ are segments of length at least $2m$. Then there is a $(K'+1)$-thin essential hallway $h\colon [-m, m]\times I\to X$ such that $h\mid [-m, m]\times \{0\}$ is a subpath of $\alpha_1$ containing $y$ and $h\mid[-m, m]\times\{1\}$ is a subpath of $\alpha_2$ containing $z$. But then since $X$ has bounded girth hallways, $h\mid \{0\}\times I$ is a path of length at most $G(K'+1)$. Since the edge maps are all quasi-isometric embeddings, this implies that there is a constant $K'''$ so that $h\mid \{j\}\times I$ is a path of length at most $K'''$ for each $j\in \{-m, \ldots, m\}$. Since one of these paths connect $y$ with $z$, we see that $d(y, z)\leqslant K'''$ as claimed.
\end{proof}

\begin{claim}
\label{claim:final_claim}
There is a constant $K''''$ such that the following holds. If $y\in p_i$, then there is a point $z\in \lambda$ so that $d(y, z)\leqslant K''''$.
\end{claim}

\begin{proof}
By \cref{claim:short_path}, $y$ lies at distance at most $K''$ from a point $y'\in p_i$ so that any transverse crossing path $\alpha$ traversing $y'$ either has an endpoint on $q_n\cup\lambda\cup p_0$ and $\pi(\alpha)$ is a segment of length at most $2m$, or $\pi(\alpha)$ is a segment of length at least $2m$. In the first instance, since $\len(p_0), \len(q_n)\leqslant L$, we have that $y'$ lies at distance at most $2m(K'+1) + L$ from a point on $\lambda$ and so $y$ lies at distance at most $K'' + 2m(K'+1) + L$ from a point on $\lambda$. In the second instance, by \cref{claim:thin_hallway} $y'$ lies at distance at most $K'''$ from a point $z'\in p_i$ (following along $p_i$ in the direction towards $\lambda$) so that there is a transverse crossing path $\beta$ traversing $z'$ with an endpoint on $q_n\cup \lambda\cup p_0$ (in fact, it must have an endpoint on $\lambda$) and such that $\pi(\beta)$ is a segment of length at most $2m$. In this case we have that $z'$ lies at distance at most $2m(K'+1) + L$ from a point on $\lambda$ and so $y$ lies at distance at most $K'' + K''' + 2m(K'+1) + L$ from a point on $\lambda$. In all cases, we may take $K'''' = K'' + K''' + 2m(K'+1) + L$ to complete the proof.
\end{proof}

We may now complete the proof of the proposition. Let $x\in \gamma$ be any point. If there is no point $y\in \lambda$ so that $d(x, y)\leqslant K'+1$, then there is a point $y\in p_i$ for some $i$ so that $d(x, y)\leqslant K' + 1$ (using the fact that a point on $\gamma_i$ is at distance at most $K'$ from a point on $q_i\cup\lambda_i\cup p_i$). \cref{claim:final_claim} then completes the proof.
\end{proof}

\begin{proposition}
\label{prop:concatenated_qg_track_1}
Let $X$ be a tree of hyperbolic spaces satisfying the qi-embedded and the bounded girth hallways condition and let $K> 0, C\geqslant 0$ be constants. There is a constant $M'\geqslant 0$ such that the following holds.

If $\lambda \colon I \to X$ is a path so that its projection to $T$ is a path without backtracking and so that each subpath of $\lambda$ in each vertex space $X_v$ is $(K, C)$-quasi-geodesic, then (K, C)-quasi-geodesics in $X$ connecting $o(\lambda)$ with $t(\lambda)$ lie in the $M'$-neighbourhood of $\lambda$.
\end{proposition}

\begin{proof}
Let $\gamma'\colon I\to X$ be a (K, C)-quasi-geodesic connecting $o(\lambda)$ with $t(\lambda)$. Now replace each maximal subpath of $\gamma'$ that projects in $T$ to a loop of positive length at a vertex $v$ with a geodesic in $X_v$ connecting the endpoints. Since each such replacement decreases the length of the path in $T$, after finitely many replacements we may obtain a new path $\gamma\colon I\to X$ so that $\pi\circ\gamma$ is a path without backtracking in $T$. By \cref{thm:bounded_hallways}, it follows that $\gamma$ is a quasi-geodesic. In particular, there is some constant $K'$ (depending on the hyperbolicity, quasi-geodesic and quasi-convexity constants) so that $\gamma'$ lies in the $K'$-neighbourhood of $\gamma$. Finally, applying \cref{prop:qg_track} we see that $\gamma'$ lies in the $(M + K')$-neighbourhood of $\lambda$.
\end{proof}

\begin{proposition}
\label{prop:concatenated_qg_track_2}
Let $X$ be a tree of relatively hyperbolic spaces satisfying the assumptions of \cref{cor:bounded_hallways} and let $K> 0, C\geqslant 0$ be constants. There is a constant $M'\geqslant 0$ such that the following holds.

If $\lambda \colon I \to \wh{X}$ is a path so that its projection away from cone points to $T$ is a path without backtracking and so that each subpath of $\lambda$ in each vertex space $\wh{X}_v$ is $(K, C)$-quasi-geodesic, then (K, C)-quasi-geodesics in $X$ connecting $o(\lambda)$ with $t(\lambda)$ lie in the $M'$-neighbourhood of $\lambda$.
\end{proposition}

\begin{proof}
As in the proof of \cref{cor:bounded_hallways}, let $\wh{\overline{X}}$ be the coning off of $\overline{X}$ with respect to the maximal cone-subtrees. Then \cite[Lemma 2.4]{MR08} states that there is a constant $K''$ such that any geodesic in $\wh{\overline{X}}$ lies in the $K''$-neighbourhood of any geodesic in $\overline{X}$ connecting the same pair of endpoints. Since the map $\wh{\overline{X}}\to \wh{X}$ given by collapsing the cones to points is a quasi-isometry, the result follows for $\wh{X}$ by applying \cref{prop:concatenated_qg_track_1} to $\overline{X}$ (which is a tree of hyperbolic metric spaces satisfying the qi-embedded and bounded girth hallways condition).
\end{proof}

\subsection{Relatively quasi-convex subgroups}

Recall that if $G$ is a group acting on a (simplicial) tree $T$ without edge inversions, then there is an associated graph of groups $\mathcal{G} = (\Gamma, \{G_v\}, \{G_e\}, \{\partial_{e}^\pm\})$ with $\Gamma\cong G\backslash T$ and with fundamental group $\pi_1(\mathcal{G})\isom G$. The reader is directed to the work of Bass \cite{Ba93} and Serre \cite{Se80} for the relevant background on graphs of groups.

We first restate \cref{thm:combination} in terms of groups using \cref{lem:action_def}.

\begin{theorem}
\label{thm:group_combination}
Let $X$ be a tree $(T)$ of relatively hyperbolic spaces such that:
\begin{enumerate}
\item $X$ satisfies the strictly type preserving, the qi-embedded and the qi-preserving electrocution condition.
\item The induced tree of coned-off spaces $\overline{X}$ satisfies the hallways flare condition and the cone-bounded hallways strictly flare condition.
\end{enumerate}
Let $G$ be a group acting geometrically on $X$ so that the projection map $\pi$ induces an action of $G$ on $T$. Then $G$ is isomorphic to the fundamental group of the quotient graph of relatively hyperbolic groups $\mathcal{G}$ and, if $\mathcal{P} = \{\stab_G(C)\}_{G\cdot C\in G\backslash\mathcal{C}}$, then the pair $(G, \mathcal{P})$ is relatively hyperbolic.
\end{theorem}

Using the results obtained in \cref{sec:quasi-geodesics} we may obtain a criterion for relative quasi-convexity of subgroups of graphs of relatively hyperbolic groups.

\begin{theorem}
\label{thm:qc_criterion}
Let $G$, $T$ and $X$ be as in \cref{thm:group_combination} and suppose that $\overline{X}$ also satisfies the bounded girth hallways condition.

Let $H\leqslant G$ be a subgroup and let $T'\subset T$ be an $H$-invariant subtree. If $H\backslash T'$ is finite and if $\stab_H(v)$ is relatively quasi-convex in $\stab_G(v)$ (with respect to the induced relatively hyperbolic structure on $\stab_G(v)$) for each vertex $v\in T$, then $H$ is relatively quasi-convex in $G$.
\end{theorem}

\begin{proof}
By \cref{cor:bounded_hallways}, there is a constant $K$ so that for each vertex $v\in T$, the subspace $\widehat{X}_v\subset \widehat{X}$ is $K$-quasi-convex. 

For each $H$-orbit of vertices $H\cdot v\subset T'$, choose an $H$-orbit of vertices $H\cdot x_v\subset \wh{X}$ where $x_v\in X_v$. For each $H$-orbit of edges $H\cdot e\subset T'$, choose $H$-orbits of vertices $H\cdot x_{e}^-, H\cdot x_{e}^+\subset \wh{X}$ where $x_e^{\pm}\in X_{e^{\pm}}$ are the images of $x_e$ under the two edge space maps. Then consider the $H$-invariant subset
\[
X_H = H\cdot \left(\bigcup_{H\cdot v\in V(H\backslash T')}x_v\right)\cup\left(\bigcup_{H\cdot e\in E(H\backslash T')}x_e^{\pm}\right)\subset \wh{X}.
\]
Since $H\backslash T'$ is finite, $H\backslash X_H$ is finite. We want to show that $X_H$ is a quasi-convex subset of $\wh{X}$. Let $K'$ be the maximum amongst all quasi-convexity constants for $X_H\cap \wh{X}_v$ over all vertices $v$. Let $x, y\in X_H$ and let $\gamma\colon I\to\wh{X}$ be a geodesic connecting $x$ with $y$. 

If $x, y$ lie in the same vertex space $\wh{X}_v$, then $\gamma$ lies in the $K$-neighbourhood of a geodesic in $\wh{X}_v$ connecting $x$ with $y$. Since $X_H\cap \wh{X}_v$ is a quasi-convex subspace by assumption, it follows that any geodesic in $\wh{X}$ connecting $x$ with $y$ lies in the $(K+K')$-neighbourhood of $X_H$.

Now suppose that $x, y$ do not lie in the same vertex space. Let $e_1*\ldots *e_n\subset T'$ be the geodesic in $T'$ connecting $\pi(x)$ with $\pi(y)$. Then choose any sequence of points $x = x_1, y_1, \ldots, x_n, y_n = y$ with $x_i\in H\cdot x_{e_{i-1}}^+\cap X_{e_{i-1}^+}$ for each $i>1$ and $y_i\in H\cdot x_{e_{i}}^-\cap X_{e_{i}^-}$ for each $i<n$. We also choose the points so that $d(y_i,x_{i+1}) = 1$ for each $i<n$. Note that each pair $x_i, y_i$ lies in the same vertex space and so we may connect them by geodesics (in the corresponding vertex spaces) and obtain a path $\lambda\colon I \to \wh{X}$ which projects to $e_1*\ldots *e_n$ and so that each subpath of $\lambda$ in each vertex space $\wh{X}_v$ is a geodesic connecting two points in $X_H$. In particular, since $X_H\cap \wh{X}_v$ is $K'$-quasi-convex in $\wh{X}_v$ for each vertex $v\in V(T)$, we see that $\lambda$ lies in the $K'$-neighbourhood of $X_H$. By \cref{prop:concatenated_qg_track_2}, $\gamma$ lies in the $M'$-neighbourhood of $\lambda$ and so $\gamma$ lies in the $(M'+K')$-neigbourhood of $X_H$. Thus, $X_H$ is quasi-convex in $\wh{X}$ and so $H$ is relatively quasi-convex in $G$ as claimed.
\end{proof}

\section{Graph pairs and presentations of mapping tori of free groups}
\label{sec:graph_pairs}

In this section we shall recap the notation and main results from the work of Feighn--Handel \cite{FH99}. Their main result is a description of particularly nice finite presentations of finitely generated subgroups of mapping tori of free groups. After describing their work, we shall prove a slight strengthening of their main result, \cref{FHmain}, and point out some useful corollaries.

\subsection{Invariant graph pairs}

Let $\F$ be a free group and let $(R, v_R)$ be a pointed graph with $\pi_1(R, v_R)$ identified with $\F$. Every graph we consider in this section will come with a pointed graph map $f_Z\colon (Z, v_Z)\to (R, v_R)$ which we shall often omit. Following \cite{FH99}, we shall denote by $Z^{\#}$ the image of $\pi_1(Z, v_Z)$ in $\pi_1(R, v_R) = \F$, induced by the graph map $f_Z$.

Let $(Z, v_Z)$ be a connected pointed graph, $f_Z\colon (Z, v_Z)\to (R, v_R)$ a graph map and let $X\subset Z$ be a connected subgraph containing the basepoint $v_Z$. We call $(Z, X)$, together with the map $f_Z$, which we shall often suppress, a \emph{graph pair}. The \emph{relative rank} of a graph pair $(Z, X)$ is 
\[
\rr(Z, X) = \rk(\pi_1(Z, v_Z)) - \rk(\pi_1(X, v_Z)),
\]
where $\rk(G)$ denotes the rank of a group $G$. When $Z- X$ consists of finitely many 0-cells and 1-cells we have $\rr(Z, X) = -\chi(Z, X)$. The graph pair $(Z, X)$ is \emph{tight} if the map $f_Z$ is an immersion. 

If $(Z', X')$ is another graph pair with $f_{Z'}\colon (Z', v_{Z'})\to (R, v_R)$ the underlying graph map, a \emph{map of graph pairs} is a pair 
\[
q = (q_Z, q_X)\colon (Z, X)\to (Z', X')
\]
of pointed graph maps $q_Z\colon (Z, v_Z)\to(Z', v_{Z'})$ and $q_X\colon (X, v_Z)\to (X', v_{Z'})$ such that $q_X = q_Z\mid X$ and so that the following diagram commutes
\[
\begin{tikzcd}
Z \arrow[rd, "f_{Z}"'] \arrow[r, "q_Z"] & Z' \arrow[d, "f_{Z'}"] \\
                                        & R                     
\end{tikzcd}
\]
The map $f_Z$ can (and often will) be considered as a map of graph pairs $(Z, X) \to (R, R)$.

Let $\psi\colon \F\to \F$ be an injective endomorphism of $\F$. A graph pair $(Z, X)$ is \emph{$\psi$-invariant} if
\[
Z^{\#} = \langle X^{\#}, \psi(X^{\#})\rangle.
\]
Now consider the group
\[
G = \F*_{\psi} = \langle \F, t \mid t^{-1}ft = \psi(t), \forall f\in \F\rangle.
\]
If $H\leqslant G$ is a subgroup, then we say that $(Z, X)$ is a \emph{$\psi$-invariant graph pair for $H$} if
\[
\langle X^{\#}, t \rangle = H.
\]
A finite $\psi$-invariant graph pair $(Z, X)$ for $H$ is \emph{minimal} if $\rr(Z, X)\leqslant \rr(Z', X')$ for all finite $\psi$-invariant graph pairs $(Z', X')$ for $H$.

\subsection{The Feighn--Handel tightening procedure}

A key part of the Feighn--Handel paper \cite{FH99} is their tightening procedure. This is a procedure which takes as input a finite $\psi$-invariant graph pair $(Z, X)$ for a subgroup $H$ and repeatedly folds the map $f_Z\colon(Z, X)\to (R, R)$, adding new loops to fix $\psi$-invariance when necessary, until a finite tight $\psi$-invariant graph pair $(\check{Z}, \check{X})$ for $H$ is obtained. Since we shall need to build on their approach, we describe their procedure in detail.

Let $(Z, X)$ be a graph pair. A fold $q_Z\colon Z\to Z_1$ induces a map of graph pairs 
\[
q = (q_Z, \check{q}_X)\colon (Z, X)\to (Z_1, X_1).
\]
The notation $\check{q}_X = q_Z\mid X$ is chosen to emphasise the fact that $\check{q}_X$ may not be a fold, even though $q_Z$ is. The induced map $\check{q}_X\colon X\to X_1$ falls into one of three cases:
\begin{enumerate}
\item $\check{q}_X$ is a fold, in which case $q$ is a \emph{subgraph fold}.
\item $\check{q}_X$ is not a fold and identifies two distinct vertices, in which case $q$ is an \emph{exceptional fold}.
\item $\check{q}_X$ is a homeomorphism.
\end{enumerate}

\begin{definition}[Folding and adding a loop if necessary]
Let $(Z, X)$ be a $\psi$-invariant graph pair for $H$, let $q_Z\colon Z\to Z_1$ be a fold and let $q = (q_Z, \check{q}_X)\colon (Z, X)\to (Z_1, X_1)$ be the induced map of pairs. If $q$ is not exceptional, then set $(Z_2, X_2) = (Z_1, X_1)$. If $q$ is exceptional, then $\check{q}_X$ identifies two distinct vertices, say $p$ and $q$. Let $\alpha, \beta\colon I\to X$ be two paths connecting the basepoint $v_Z$ with $p$ and $q$ respectively. Then let $\delta=[(z\circ\alpha)*(z\circ\overline{\beta})]\in \F$. If $\psi(\delta)\in Z_1^{\#}$, then set $(Z_2, X_2) = (Z_1, X_1)$. If not, then set $(Z_2, X_2) = (Z_2\vee\Gamma(\delta), X_1)$. We say that the pair $(Z_2, X_2)$ is obtained from $(Z, X)$ by \emph{folding and adding a loop if necessary}.
\end{definition}

The lemma below is \cite[Lemma 4.7]{FH99}.

\begin{lemma}
\label{lem:graph_pair_faalin}
If $(Z, X)$ is a $\psi$-invariant graph pair for $H$ and if $(Z_2, X_2)$ is obtained from $(Z, X)$ by folding and adding a loop if necessary, then $(Z_2, X_2)$ is also a $\psi$-invariant graph pair for $H$ and $\rr(Z_2, X_2)\leqslant \rr(Z, X)$.

If $(Z, X)$ factors through a tight $\psi$-invariant graph pair $(Z', X')$, then so does $(Z_2, X_2)$.
\end{lemma}

\begin{definition}[Tightening]
Let $(Z, X)$ be a finite $\psi$-invariant graph pair for $H$. If $(Z, X)$ is tight then do nothing. If $X$ is not tight, then perform a subgraph fold. If $X$ is tight, but $Z$ is not, then fold and add a loop if necessary. Repeat until we are left with a tight $\psi$-invariant graph pair $(\check{Z}, \check{X})$ for $H$. Say $(\check{Z}, \check{X})$ is obtained from $(Z, X)$ by \emph{tightening}.
\end{definition}

The lemma below is stated in \cite[Definition 4.6]{FH99}.

\begin{lemma}
\label{lem:FH_tightening}
If $(Z, X)$ is a finite $\psi$-invariant graph pair for $H$, the tightening procedure terminates after finitely many steps having produced a finite tight $\psi$-invariant graph pair $(\check{Z}, \check{X})$ for $H$.

If $(Z, X)$ factors through a tight $\psi$-invariant graph pair $(Z', X')$, then so does $(\check{Z}, \check{X})$.
\end{lemma}

The following is the key result in \cite[Proposition 5.4]{FH99}.

\begin{proposition}
\label{prop:FH}
Let $(Z, X)$ be a finite $\psi$-invariant graph pair for $H$ with $(f_X)_*$ injective, but $(f_Z)_*$ not injective. If $(\check{Z}, \check{X})$ is obtained from $(Z, X)$ by tightening, then $\rr(\check{Z}, \check{X})<\rr(Z, X)$.
\end{proposition}

\subsection{Finitely generated subgroups of mapping tori of free groups}

Using \cref{prop:FH} we may obtain a criterion for minimality. The proof is essentially the proof of the main proposition in \cite{FH99}. Note that \cref{itm:pres} in \cref{FHmain} is precisely Feighn--Handel's main proposition.

\begin{theorem}
\label{FHmain}
Let $(Z, X)$ be a finite $\psi$-invariant graph pair for $H$ with $(f_Z)_*$ injective and let $C\leqslant Z^{\#}$ so that $Z^{\#} = X^{\#}*C$. The following are equivalent:
\begin{enumerate}
\item\label{itm:min} The pair $(Z, X)$ is minimal.
\item\label{itm:inj} The map
\[
\theta_n\colon\pi_1\left(X\vee \bigvee_{i=0}^n\Gamma(\psi^i(C)), v_X\right) \to \pi_1(R, v_R)
\]
is injective for all $n\geqslant 0$.
\item\label{itm:pres} We have
\[
H\isom \langle Z^{\#}, t \mid t^{-1}xt = \psi(x), \forall x\in X^{\#}\rangle.
\]
\end{enumerate}
In particular, if any of the above hold, then $\chi(H) = -\rr(Z, X)$.
\end{theorem}

\begin{proof}
If $(Z, X)$ is minimal, then the map $\theta_n$ injective for all $n\geqslant 0$ by \cref{prop:FH}. Hence, \cref{itm:min} implies \cref{itm:inj}.

Now suppose that $\theta_n$ is injective for all $n$. Then, we may identify the group
\[
L = X^{\#}*\Asterisk_{i=0}^{\infty}\psi^i(C)
\]
with a subgroup of $\F = \pi_1(R, v_R)$. In particular, since $C\leqslant H$ and $H = \langle X^{\#}, t\rangle$, we see that $H = \langle L, t\rangle$. We claim that the homomorphism
\[
\phi\colon \langle X^{\#}, C, t \mid t^{-1}ft = \psi(f), \forall f\in X^{\#}\rangle = \overline{H} \to G
\]
is injective. Since $\phi(\overline{H}) = H$, this will imply that \cref{itm:inj} implies \cref{itm:pres}.

The kernel of the epimorphism $\lambda\colon\overline{H}\to \Z$ given by quotienting by the normal closure of $Z^{\#} = X^{\#}*C$ is isomorphic to
\[
\ker(\lambda) \isom \ldots \underset{X^{\#}}{*}Z^{\#}\underset{t^{-1}X^{\#}t}{*}t^{-1}Z^{\#}t\underset{t^{-2}X^{\#}t^2}{*}\ldots
\]
For $i\geqslant 0$, denote by
\[
\F_{i} = \langle Z^{\#}, t^{-1}Z^{\#}t, \ldots, t^{-i}Z^{\#}t^i\rangle.
\]
and let $\F_{\infty} = \bigcup_{i=0}^{\infty}\F_i$. Using the decomposition of $\ker(\lambda)$ above we see that for each $i\geqslant 0$ we have
\begin{align*}
\F_{i} &= \F_{i-1}\underset{t^{-i}X^{\#}t^i}{*}t^{-i}Z^{\#}t^i\\
		&= \F_{i-1}\underset{t^{-i}X^{\#}t^i}{*}(t^{-i}X^{\#}t^i*t^{-i}Ct^i)\\\
		&= \F_{i, j-1}*t^{-i}Ct^i.
\end{align*}
It follows by induction that we have
\[
\F_{\infty} = X^{\#} * C* \ldots * t^{-i}Ct^i*\ldots
\]
Hence, the homomorphism $\phi\mid \F_{\infty}$ is an isomorphism onto $L$. Since every element in $\ker(\lambda)$ is conjugate into $\F_{\infty}$, it follows that any non-trivial element in $\ker(\phi)$ has non-zero exponent sum on $t$. This is not possible and so $\ker(\phi) = 1$, proving our claim. In particular, \cref{itm:inj} implies \cref{itm:pres}.

Assuming \cref{itm:pres}, by a result of Chiswell--Collins--Huebschmann \cite{CCH81}, we have that $-\chi(H) = \rk(C) = \rr(Z, X)$. Since $\chi(H)$ is a group invariant, the value $\rr(Z, X)$ depends only on $H$. Thus, since \cref{itm:min} implies \cref{itm:pres}, we see that $\rr(Z, X)$ is minimal and so $(Z, X)$ is minimal. This completes the proof.
\end{proof}

We conclude this section with some auxiliary facts about mapping tori of free groups which follow from the work of Feighn--Handel.

\begin{corollary}
\label{cor:free_prod_decomp1}
Let $\F$ be a free group, let $\psi\colon \F\to \F$ be a monomorphism and let $G=M(\psi)$ be the mapping torus. Then $G$ is finitely generated if and only if there is a $\psi$-invariant subgroup
\[
\F' = A*\left(\Asterisk_{i\geqslant 0}C_i\right)\leqslant \F
\]
where $A$ and $C_0$ are finitely generated, where $C_i = \psi^i(C_0)$ for each $i\geqslant 0$ and so that $G\cong M(\phi)$ (induced by the inclusion) where $\phi = \psi\mid\F'$. 
\end{corollary}

\begin{proof}
This is a direct application of \cref{FHmain}.
\end{proof}

If $G$ is a group and $h\in G$ is an element, the conjugation (by $h$) homomorphism, denoted by $\gamma_h\colon G\to G$ is given by $g\mapsto h^{-1}gh$.

\begin{corollary}
\label{cor:Euler}
Let $\psi\colon \F\to \F$ be a monomorphism and let $G = \F*_{\psi}$. If $H\leqslant \F$ is a finitely generated subgroup so that $\psi^k(H)\leqslant H^f$ for some $k\geqslant 1$ and $f\in \F$, then $\langle H, t^kf\rangle \cong M(\phi)$ where $\phi = \gamma_f\circ\psi^k\mid H$.
\end{corollary}

\begin{proof}
After replacing $\psi$ with the monomorphism $\gamma_f\circ\psi^k\colon \F\to \F$, we see that $(\Gamma(H), \Gamma(H))$ is a tight $\psi$-invariant graph pair for $\langle H, t^kf\rangle$. Now $\langle H, t^kf\rangle$ has the required presentation by \cref{FHmain}.
\end{proof}

\begin{remark}
The subgroup $\langle H, t^kf\rangle\leqslant M(\psi)$ from \cref{cor:Euler} is sometimes referred to a \emph{sub-mapping torus}. It is a consequence of \cref{FHmain} that every non-cyclic subgroup $H\leqslant M(\psi)$ with $\chi(H) = 0$ is conjugate to a sub-mapping torus.
\end{remark}

\begin{corollary}
\label{fg_kernel} 
Let $\F$ be a free group and let $F_n$ be the free group of rank $n\geqslant 1$. If $\psi\colon \F\to \F$ is an isomorphism, $\phi\colon F_n\to F_n$ is a monomorphism and if $\F\rtimes_{\psi}\Z\isom M(\phi)$, then $\F$ is finitely generated.
\end{corollary}

\begin{proof}
Let $G = \F\rtimes_{\psi}\Z$. If $G\isom M(\phi)$ for some $n\geqslant 1$ and some injective endomorphism $\phi$, then $\chi(G) = 0$ by \cref{FHmain}. Also by \cref{FHmain}, we have $\chi(G) = \rr(Z, X)$ for some finite $\psi$-invariant graph pair $(Z, X)$. Hence, $\rr(Z, X) = 0$ and so $\psi(X^{\#}) = \psi(Z^{\#})\leqslant Z^{\#}$. If $\psi\mid Z^{\#}$ is not surjective on $Z^{\#}$, then $Z^{\#}\leqslant \psi^{-1}(Z^{\#})\leqslant \psi^{-2}(Z^{\#})\leqslant\ldots\leqslant \F$ form a proper ascending chain of subgroups of a free group of fixed rank. But this contradicts Takahasi's Theorem \cite{Ta51} and so $\F = Z^{\#}$ as claimed.
\end{proof}

\section{The peripheral subgroups: maximal sub-mapping tori}
\label{sec:peripherals}

In this section, we show that any finitely generated mapping torus of a free group has a canonical collection of (conjugacy classes of) maximal subgroups that are sub-mapping tori of finitely generated free groups. This collection of sub-mapping tori will be the peripheral subgroups for the relatively hyperbolic structure from \cref{main}.

Using the decomposition from \cref{cor:free_prod_decomp1}, we first explain how to obtain a natural splitting of a finitely generated mapping torus $M(\psi)$ as a HNN-extension over a finitely generated free group.

Assume the notation from \cref{cor:free_prod_decomp1}. Let $m\geqslant 0$ be large enough so that
\[
\psi(A)\leqslant A*(\Asterisk_{i=0}^m\psi^i(C)).
\]
Such an $m$ exists since $A$ is finitely generated. Then, denoting by $C_i = \psi^i(C)$ for $0\leqslant i\leqslant m$, define
\[
\phi\colon A*(\Asterisk_{i=0}^{m-1}C_i) = L \to U =  \psi(A)*(\Asterisk_{i=1}^mC_i)
\]
to be the isomorphism given by $\psi\mid L$. Denoting by
\[
F = A*(\Asterisk_{i=0}^mC_i),
\]
it is not hard to see that we have:
\begin{align}
\label{eq:decomposition}
M(\psi) &\isom F*_{\phi}.
\end{align}
We now analyse the action of $M(\psi)$ on the Bass--Serre tree associated with the decomposition \eqref{eq:decomposition}.

\begin{theorem}
\label{thm:acylindrical}
Let $F = A*(\Asterisk_{i=0}^mC_i)$ be a finitely generated free group and let 
\[
\phi\colon A*(\Asterisk_{i=0}^{m-1}C_i) = L \to U = \psi(A)*(\Asterisk_{i=1}^mC_i)
\]
be an isomorphism so that $\phi(C_i) = C_{i+1}$ for each $0\leqslant i<n$. There exists a free product decomposition
\[
F = A_1*\ldots *A_n*B*\left(\Asterisk_{i = 0}^mC_i\right)
\]
so that the following holds.
\begin{enumerate}
\item There is a map $\sigma\colon \{1, \ldots, n\}\to\{1, \ldots, n\}$ such that for each $1\leqslant i\leqslant n$, there is an $f_i\in F$ so that $\psi(A_i)^{f_i} \leqslant A_{\sigma(i)}$. Moreover, there can be no $1\leqslant i< j\leqslant n$ with $\sigma(i) = \sigma(j)$ and $f_i, f_j\in L$ so that $\psi(A_i^{f_i}), \psi(A_j^{f_j})\leqslant A_{\sigma(i)}$.
\item If $T$ is the Bass--Serre tree for the HNN-extension $G = F*_{\phi}$, then there exists an integer $k\geqslant 0$ such that for any subset $S\subset T$ containing at least $k$ vertices, the pointwise stabiliser $\stab(S)$ of $S$ is conjugate into some $A_i$.
\item\label{itm:submapping_tori} For each $1\leqslant i\leqslant n$ so that $\sigma^{\ell_i}(i) = i$ for some $\ell_i\geqslant 1$ (assume $\ell_i$ is the minimal such integer), we have $A_{i}^{h_i}\leqslant A_{i}$ where 
\[
h_i = tf_{i}tf_{\sigma(i)}\ldots tf_{\sigma^{\ell_i-1}(i)}.
\]
In particular, $H_i = \langle A_{i}, h_{i}\rangle$ is isomorphic to a mapping torus of the finitely generated free group $A_{i}$.
\item If $H\leqslant G$ is a isomorphic to a mapping torus of a finitely generated free group, then $H$ is conjugate into some $H_i$ as in \cref{itm:submapping_tori}.
\end{enumerate}
\end{theorem}

\begin{proof}
We first note that
\[
\langle \{A, t^{-i}C_0t^i\}_{i\geqslant 0}\rangle \cong A*(\Asterisk_{i\geqslant 0}t^{-i}C_0t^i).
\]
Denote by
\[
\F = A*(\Asterisk_{i\geqslant 0}C_i)
\]
where here we identify $C_i$ with $t^{-i}C_0t^i$ for each $i\geqslant 0$. Importantly, $F$ here is a free factor of $\F$. Let $\psi\colon \F\to \F$ be given by $\phi\mid F$ and by identifying $C_i$ with $C_{i+1}$. We have
\[
F*_{\phi} \cong M(\psi)
\]
Let $T$ be the Bass--Serre tree for the splitting $F*_{\phi}$. Recall that the vertices of $T$ are cosets $gF$ of $F$ in $F*_{\phi}$. There is a natural orientation on the edges of $T$ induced by a choice of orientation on the single edge in $F*_{\phi}\backslash T$.

By Mutanguha \cite[Proposition 5.3.1]{Mu21}, there is a constant $\kappa>0$ and a free factor system $\mathcal{F}$ of $F$ so that the following holds. For each $H\in \mathcal{F}$, $\psi(H)$ is conjugate into a free factor in $\mathcal{F}$ and so that for every finitely generated subgroup $H\leqslant F$ such that $\psi^n(H)$ is conjugate into $F$ (within $\F$) for some $n\geqslant \kappa$, we have that $H$ is conjugate into a free factor in $\mathcal{F}$. Then since $\F = F*(\Asterisk_{i\geqslant 0}\psi^{m+1}(C_i))$, if $\kappa>m$ we see that
\[
\{\psi^{\kappa}(H)\}_{H\in \mathcal{F}}\cup \{\psi^{\kappa}(C_i)\}_{i=0}^m
\]
is a free factor system of $\F$. Thus, after possibly increasing $\kappa$ if necessary,
\[
\mathcal{F}\cup \{C_i\}_{i=0}^m
\]
is a free factor system of $F$. In particular, we have
\[
F = A_1*\ldots*A_n*B*(\Asterisk_{i=0}^mC_i),
\]
where $\mathcal{F} = \{A_i\}_{i=1}^n$ (after possibly replacing groups in $\mathcal{F}$ with conjugates) and there is a map $\sigma\colon \{1, \ldots, n\}\to \{1, \ldots, n\}$ and elements $f_i\in F$ such that $\psi(A_i)^{f_i} \leqslant A_{\sigma(i)}$ for each $i$. 

If there is some $i\neq j$ such that $\sigma(i) = \sigma(j)$ and such that $\psi(A_i^{f_i}), \psi(A_j^{f_j})\leqslant A_{\sigma(i)}$ for some $f_i, f_j\in L$, then $\psi^m(\langle A_i^{f_i}, A_j^{f_j}\rangle)$ would be conjugate into $F$ for all $m\geqslant 0$. But this would imply that $\langle A_i^{f_i}, A_j^{f_j}\rangle$ would have to be conjugate into some free factor in $\mathcal{F}$ which is not possible. Thus, there can be no $1\leqslant i< j\leqslant n$ with $\sigma(i) = \sigma(j)$ and $f_i, f_j\in L$ so that $\psi(A_i^{f_i}), \psi(A_j^{f_j})\leqslant A_{\sigma(i)}$.

By definition of $\mathcal{F}$, the segments in the Basse--Serre tree $T$ of length at least $\kappa$ that follow the induced orientation have stabiliser conjugate into some $A_i\in \mathcal{F}$ within $F$. Now let $S\subset T$ be a subset containing at least $k$ vertices, where we set $k = 2\kappa$. The pointwise stabiliser of $S$ also fixes the convex closure of $S$ so we may assume that $S$ is convex. We may also assume that $S$ is compact.

If $S$ contains a pair of distinct edges leading out of a common vertex (according to the induced orientation), then we claim that $\stab(S) = 1$. Indeed, if $e_1, e_2$ are the two edges, we see that (after possibly translating them so they lead out of the vertex $F$) $\stab(e_1) = g_1^{-1}L$ and $\stab(e_2) = g_2^{-1}L$ for some elements $g_1, g_2\in F$ in distinct $L$ cosets (in $F$). Since $L$ is a free factor of $F$, it is malnormal by \cref{ff_malnormal}. Since $\stab(e_1\cup e_2) = L^{g_1}\cap L^{g_2}$, we see that $\stab(e_1\cup e_2) = 1$. Hence also $\stab(S) = 1$.

If $S$ does not contain a pair of distinct oriented edges leading out of a common vertex, then it is either a single geodesic segment of length $\geqslant 2\kappa$ following the induced orientation, or it is a union of two geodesic segments following the induced orientation which intersect in a common terminal segment and such that one of the two segments has length at least $\kappa$. Hence, without loss of generality, we may assume that $S$ is actually a geodesic segment of length at least $\kappa$. But we already showed that $\stab(S)$ is conjugate into $A_i$ for some $1\leqslant i\leqslant n$.

Now for the final statements. The fact that $A_{i}^{h_i}\leqslant A_{i}$ for each $i$ follows from the definition of the $h_i$. Hence, the subgroup $H_i = \langle A_{i}, h_i\rangle$ is isomorphic to a mapping torus of a finitely generated non-trivial free group with $H_i/ \normal{A_{i}} \isom \Z$ by \cref{cor:Euler}. Now let $H\leqslant F*_{\phi}$ be isomorphic to a mapping torus of a finitely generated (non-trivial) free group. Since $H$ is not free, we have that $H$ has non-trivial image in $F*_{\phi}/\normal{F} \isom \Z$. By \cref{cor:Euler}, after possibly replacing $H$ with a conjugate, there is a finitely generated subgroup $F'\leqslant \F$, an element $f\in \F$ and an integer $j$ so that $\psi^j(F')^f\leqslant F'$ and $H = \langle F', t^jf\rangle$. By possibly increasing $m$ if necessary, we may assume that $F'\leqslant F$. Then $F'$ stabilises the axis for $t^jf$ in the Bass--Serre tree for $F*_{\phi}$. We showed that this implies that $F'$ must be conjugate into some $A_i$. After replacing $H$ with a conjugate again, we may assume that $F'\leqslant A_i$. Now since $F'^{t^jf}\leqslant F'$, this also implies that $\psi^j(A_i)$ is conjugate into $A_i$ and so that $\sigma^j(i) = i$. Let $\ell\geqslant 2$ be minimal so that $i_1 = i, \ldots, i_{\ell} = \sigma^{\ell-1}(i) = i$. We see that $\ell$ must divide $j$ and so $h_i^p = t^jfg$ for some $g\in \F$ and where $p = \frac{j}{\ell}$. Since $A_i$ is a free factor of $\F$, it is malnormal by \cref{ff_malnormal}. Hence, since $F'^{t^jf} \leqslant A_i$ and $F'^{h_i^p} = F'^{t^jfg}\leqslant A_i$, we see that $g\in A_i$. This thus implies that $H\leqslant H_i$ and we are done.
\end{proof}

\begin{remark}
The condition of $H\leqslant G$ being isomorphic to a mapping torus of a finitely generated non-trivial free group is equivalent to $H$ not being cyclic and $\chi(H) = 0$ by \cref{FHmain}.
\end{remark}

Note that it is not true that each $1\leqslant i\leqslant n$ in \cref{thm:acylindrical} gives rise to a mapping torus of a finitely generated free group $H_i$; it is only the indices $i$ so that $\sigma^j(i) = i$ for some $j\geqslant 2$. Note also that any pair of indices $1\leqslant i, j\leqslant n$ that lie in the same $\sigma$(-periodic) orbit give rise to conjugate mapping tori of finitely generated free groups.

Applying \cref{thm:acylindrical} to the splitting \eqref{eq:decomposition} we obtain \cref{cor:canonical_collection} which is the first statement in \cref{main}.

\begin{corollary}
\label{cor:canonical_collection}
Let $\F$ be a free group and let $\psi\colon \F\to \F$ be a monomorphism so that the mapping torus $M(\psi)$ is finitely generated. There is a (possibly empty) finite collection of (conjugacy classes of) subgroups $\mathcal{P}$ of $M(\psi)$, each isomorphic to a mapping torus of a finitely generated free group so that if $H\leqslant M(\psi)$ is isomorphic to a mapping torus of a finitely generated non-trivial free group, then $H$ is conjugate into a unique $P\in\mathcal{P}$.
\end{corollary}

\begin{remark}
The (conjugacy classes of) subgroups $\mathcal{P}$ from \cref{cor:canonical_collection} form a malnormal collection (after removing repeats) and so are canonical.
\end{remark}

\section{Relative hyperbolicity of the mapping torus}
\label{sec:rel_hyp_mapping_tori}

The aim of this section is to establish the second statement from \cref{main}.

\begin{theorem}
\label{thm:rel_hyp}
Let $\F$ be a free group and let $\psi\colon \F\to \F$ be a monomorphism so that the mapping torus $M(\psi)$ is finitely generated. Then $(M(\psi), \mathcal{P})$ is relatively hyperbolic, where here $\mathcal{P}$ is the canonical collection of maximal sub-mapping tori of finitely generated free groups from \cref{cor:canonical_collection}.
\end{theorem}

\begin{remark}
If in \cref{thm:rel_hyp} $\psi$ is an automorphism, then $M(\psi)\cong \F\rtimes_{\psi}\Z$ is free-by-$\Z$ and each $P\in \mathcal{P}$ will be \{finitely generated free\}-by-$\Z$  by \cref{fg_kernel}.
\end{remark}

In order to prove \cref{thm:rel_hyp}, we will verify that all the conditions in the Mj--Reeves combination theorem are satisfied for a certain partial mapping torus constructed from the splitting from \cref{thm:acylindrical}.

\subsection{A (partial) mapping torus}
\label{sec:mapping_torus}

If $X$ is a space, $Y\subset X$ is a subspace and $f\colon Y\to X$ is a map, then the \emph{partial mapping torus} $M(f)$ of $f$ is the space
\[
M(f) = X\sqcup \left(Y\times[-1, 1]\right)/ \{y \sim (y, -1), f(y) \sim (y, 1) \mid \forall y\in Y\}.
\]
Note that this is a graph of spaces with underlying graph with a single vertex and a single edge. When $X = Y$, this is the usual definition of the \emph{mapping torus} of $f$.

If $f$ is a cellular map of graphs, then $M(f)$ has a natural combinatorial 2-complex structure obtained from $X$ by attaching 1-cells $t_x$ connecting each 0-cell $x\in X$ with $f(x)$ and attaching 2-cells $c_e$ for each 1-cell $e\subset X$ with attaching map given by the loop $e*t_{e^+}*\overline{f(e)}*\overline{t}_{e^-}$. We now describe a (partial) mapping torus of graphs which we shall work with for the remainder of this section. We shall always assume that our (partial) mapping tori are endowed with such a combinatorial 2-complex structure.

\subsection*{The base space}

Let $G = M(\psi)$ be a finitely generated mapping torus of a free group. By \cref{cor:free_prod_decomp1,thm:acylindrical} we may assume that $\psi\colon \F \to \F$ is a monomorphism so that
\[
\F = A_1*\ldots*A_n*B*\left(\Asterisk_{i\geqslant 0}C_i\right),
\]
with $B$ and each $A_i$, $C_i$ finitely generated, $\psi(C_i) = \psi(C_{i+1})$ for each $i\geqslant 0$ and there is some map $\sigma\colon\{1, \ldots, n\}\to\{1, \ldots, n\}$ and elements $f_i\in \F$ so that $\psi(A_i) \leqslant A_{\sigma(i)}^{f_i}$ for each $1\leqslant i\leqslant n$.

Choose a free basis $\mathcal{A}_i$ for each $A_i$, a free basis $\mathcal{B}$ for $B$ and a free basis $\mathcal{C}_0$ for $C_0$. Let $\mathcal{C}_i = \psi^i(\mathcal{C}_0)$ for each $i\geqslant 0$. The set
\[
\mathcal{F} = \left(\bigsqcup_{i=1}^n\mathcal{A}_i\right)\sqcup\mathcal{B}\sqcup\left(\bigsqcup_{j\in \N}\mathcal{C}_i\right)
\]
is therefore a free basis for $\F$.

Let $R_{A_i}$ be the rose graph with a petal for each free generator in $\mathcal{A}_i$, let $R_B$ be the rose graph with a petal for each free generator in  $\mathcal{B}$ and let $R_{C_i}$ be the rose graph with a petal for each free generator in $\mathcal{C}_i$. 

Let $R$ be the graph obtained from 
\[
\left(\bigsqcup_{i=1}^nR_{A_i}\right)\sqcup R_{B}\vee\left(\bigvee_{j\in \N}R_{C_j}\right)
\]
by adding an edge $r_i$ for each $1\leqslant i\leqslant n$ connecting the vertex $v\in R_{B}\vee\left(\bigvee_{j\in \N}R_{C_i}\right)$ with the vertex $v_i\in R_{A_i}$. There is a natural identification
\[
\pi_1(R, v) \isom \F.
\]

\subsection*{The mapping torus}

We are going to define a map $f\colon R\to R$ such that $f_* = \psi$.
\begin{enumerate}
\item For each $1\leqslant i\leqslant n$, let $f_i\in F$ be an element so that $\psi(A_i) \leqslant A_{\sigma(i)}^{f_i}$ and denote by $p_i\colon I\immerses X$ the immersed loop such that $[p_i] = f_i$.
\item For each $1\leqslant i\leqslant n$ and each $g\in \mathcal{A}_i$, denote by $q_g\colon I\immerses R_{A_{\sigma(i)}}$ the immersed loop such that $[q_g] = \psi(g)^{f_i^{-1}}$ (which lies in $A_{\sigma(i)}$ by definition of the $f_i$).
\item For all $g\in \mathcal{B}\sqcup\left(\bigsqcup_{j\in \N}\mathcal{C}_j\right)$, denote by $q_g\colon I\immerses X$ the immersed loop such that $[q_g] = \psi(g)$.
\end{enumerate}
Then we define $f$ by:
\begin{align*}
f(v), f(v_i) &= v, v_{i}  & \text{ for all $1\leqslant i\leqslant n$}\\
f(r_i) &= p_i*r_{\sigma(i)} & \text{ for all $1\leqslant i\leqslant n$}\\
f(g) &= q_g & \text{ for all $g\in \mathcal{F}$}
\end{align*}
By construction, we have that $f_* = \psi$ and so
\[
\pi_1(M(f), v) \isom M(\psi).
\]

\subsection*{The partial mapping tori}

For each $l\geqslant 0$, denote by $R_l\subset R$ the subgraph obtained by removing all edges in $\bigvee_{j>l}R_{C_j}$. There is a constant $\mu\geqslant 0$ so that for all $l\geqslant \mu$ we have that
\begin{align*}
f(R_{l-1}) &\subseteq R_{l}\\
f_*(\pi_1(R_{l-1}, v)) &\leqslant \pi_1(R_{l}, v).
\end{align*}
For each $l\geqslant \mu$, denote by $M_{l}\subset M(f)$ the partial mapping torus of $f\mid R_{l-1}$ with base space $R_{l}$. This will be the (compact) space we shall work with. Note that we have
\[
\pi_1(M_l, v) \cong \pi_1(R_{l}, v)*_{\phi_l}
\]
where $\phi_l = \psi\mid\pi_1(R_{l-1}, v)$.

\subsection*{The peripheral subcomplexes}

Note that for each $1\leqslant i\leqslant n$ we have
\[
f(A_i)\subset A_{\sigma(i)}
\]
by construction. Let $\sim$ be the equivalence relation on the set $\{1, \ldots, n\}$ generated by $i\sim \sigma(i)$. Note that for each equivalence class $[i]$, there is a sequence $i_1, i_2, \ldots, i_{\ell}\in [i]$ so that $i_1 = \sigma(i_{\ell})$ and $i_j = \sigma(i_{j-1})$ for $2\leqslant j\leqslant \ell$. Moreover, for each $i\in [i]$, there is an integer $m_i\geqslant 1$ so that $\sigma^{m_i}(i) = i_1$. Denote by $M[i]\subset M(f)$ the maximal subcomplex so that
\[
M[i] \cap R = \bigcup_{i\in [i]} R_{A_i}.
\]
Note that $f^{\ell}(A_{i_1})\subset A_{i_1}$ and $M[i]$ is homotopy equivalent to the mapping torus $M(f^{\ell}\mid R_{A_{i_1}})$. Thus, we have
\begin{equation}
\label{eq:submapping}
\pi_1(M(f^{\ell}\mid R_{A_{i_1}}), v_{i_1})\cong \pi_1(M[i], v_i)\in \mathcal{P}
\end{equation}
where $\mathcal{P}$ is the collection of subgroups from \cref{cor:canonical_collection}.

\subsection*{Some facts and some constants}

We collect some essential facts about the mapping torus, the partial mapping tori and the peripheral subcomplexes.

\begin{lemma}
\label{lem:submapping_torus}
The following properties hold for all $l\geqslant \mu$:
\begin{enumerate}
\item\label{itm:incl_iso} The inclusion
\[
M_{l} \injects M(f)
\]
induces an isomorphism on fundamental groups and so $\pi_1(M_l) \cong M(\psi)$.
\item\label{itm:incl_P} For each $1\leqslant i\leqslant n$, the inclusion
\[
M[i]\injects M_l \subset M(f)
\]
induces an injection on $\pi_1$ and
\[
\pi_1(M[i], v_i)\in \mathcal{P}
\]
where $\mathcal{P}$ is the collection of subgroups from \cref{cor:canonical_collection}.
\item\label{itm:qi1} Lifts $\wt{R}_{l-1}\to \wt{R}_{l}$ of $f\mid R_{l-1}$ to the universal covers are quasi-isometric embeddings.
\item\label{itm:qi2} Lifts $\wt{f}\colon \wt{R}\to\wt{R}$ of $f$ to the universal cover $\wt{R}$ are quasi-isometric embeddings.
\end{enumerate}
\end{lemma}

\begin{proof}
\cref{itm:incl_iso} holds by definition of $M(f)$ and $M_l$.

\cref{itm:incl_P} holds by \eqref{eq:submapping} and \cref{cor:Euler}.

\cref{itm:qi1} follows from the fact that $R_{l-1}, R_{l}$ are compact graphs.

Now we prove \cref{itm:qi2}. Let $\lambda = \lambda_0*\gamma_1*\lambda_1*\ldots*\gamma_m*\lambda_m$ be a geodesic in $\widetilde{R}$ where the $\gamma_i$ are maximal subpaths which do not traverse edges in any lift $\widetilde{R}_{l-1}\injects \widetilde{R}$. In other words, each $\gamma_i$ does not traverse any edges which project to $R_{C_j}$ for any $j\geqslant l$. Letting $\lambda_i'$ be the geodesic in $\widetilde{R}$ connecting the endpoints of $\widetilde{f}(\lambda_i)$, we see that the path $\lambda_0'*\widetilde{f}(\gamma_1)*\lambda_1'*\ldots*\widetilde{f}(\gamma_m)*\lambda_m'$ is a geodesic. Since each $\lambda_i'$ must lie in a copy of $\widetilde{R}_{l}$, we see that $\widetilde{f}$ is a quasi-isometric embedding precisely if the restriction $\widetilde{R}_{l-1}\to \widetilde{R}_{l}$ is. Since $\widetilde{R}_{l-1}\to \widetilde{R}_{l}$ is a quasi-isometric embedding by \cref{itm:qi1}, $\widetilde{f}$ is a quasi-isometric embedding.
\end{proof}

We now fix some constants for the rest of the section:
\begin{enumerate}
\item $l$ is any integer greater than $\mu$.
\item $k$ is the constant from \cref{thm:acylindrical} when applied to the splitting $\pi_1(R_{l})*_{\phi_l}\cong M(\psi)$.
\item $K>0, C\geqslant 0$ are constants so that $\wt{f}$ is a $(K, C)$-quasi-isometric embedding (which exist by \cref{lem:submapping_torus}).
\end{enumerate}

\subsection{The tree of relatively hyperbolic spaces}

Consider $M_{l}\subset M(f)$ and let $\mathfrak{p}\colon\widetilde{M}_{l}\to M_{l}$ denote its universal cover. Since $M_{l}$ has the structure of a graph of spaces, $\widetilde{M}_{l}$ has the structure of a tree of spaces
\[
(T, \{X_v\}_{v\in V(T)}, \{X_e\}_{e\in E(T)}, \{\partial_e^{\pm}\}_{e\in E(T)})
\]
 with underlying tree the Bass--Serre tree $T$ for the splitting $\pi_1(M_l)\cong F*_{\phi}$, where
\[
F = A_1*\ldots *A_n*B*\left(\Asterisk_{0\leqslant i\leqslant l}C_i\right) \cong \pi_1(R_{l}, v),
\]
with each vertex space $X_v$ isomorphic to the universal cover $\widetilde{R}_{l}$ of $R_{l}$ and with each edge space $X_e$ isomorphic to the universal cover $\widetilde{R}_{l-1}$ of $R_{l-1}$. The edges of $T$ have a natural orientation given by the action of the stable letter $t$ of the HNN-extension $F*_{\phi}$ on $T$. The edges in $\widetilde{M}_{l}$ that project to edges in the tree $T$ inherit an orientation so that they connect copies of $\widetilde{R}_{l-1}$ with their images under $\widetilde{f}$.

For each edge $e\in R_{l}$, we may metrise the 2-cell $c_e$ appropriately so that the boundary path $e*t_{e^+}*\overline{f(e)}*\overline{t}_{e^-}$ has the desired length $3 + \len(f(e))$. This naturally makes $\wt{M}_l$ a metric space. Technically, in order to ensure $\wt{M}_l$ has all the properties required to be a tree of relatively hyperbolic spaces, we should ensure that a neighbourhood of each edge space $X_e$ is isometric to $X_e\times(0, 1)$, but this is not important for the proofs.

\begin{lemma}
\label{lem:conditions}
The tree of spaces $\widetilde{M}_{l}$ is a tree of relatively hyperbolic spaces, with vertex and edge pairs
\begin{align*}
(X_v, &\{X_v\cap \mathfrak{p}^{-1}(R_{A_i})\}_{i=1}^n)\\
(X_e, &\{X_e\cap \mathfrak{p}^{-1}(R_{A_i})\}_{i=1}^n).
\end{align*}
Moreover, $\wt{M}_l$ satisfies the strictly type preserving, the qi-embedded and the qi-preserving electrocution condition.
\end{lemma}

\begin{proof}
By \cref{lem:tree_rel_hyp}, for each $v\in V(T)$ and $e\in E(T)$, the pairs
\begin{align*}
(X_v, &\{X_v\cap \mathfrak{p}^{-1}(R_{A_i})\}_{i=1}^n)\\
(X_e, &\{X_e\cap \mathfrak{p}^{-1}(R_{A_i})\}_{i=1}^n)
\end{align*}
are relatively hyperbolic and so $\wt{M}_l$ is a tree of relatively hyperbolic spaces. The tree of spaces $\wt{M}_l$ satisfies the strictly type preserving condition by definition of $f$ and by \cref{thm:acylindrical}. Finally, the tree of relatively hyperbolic spaces $\wt{M}_{l}$ satisfies the qi-embedded and the qi-preserving electrocution condition by \cref{lem:submapping_torus}.
\end{proof}

\subsection{Hallways flare}

In view of \cref{lem:conditions}, it remains to verify the hallways flare and cone-bounded hallways strictly flare properties for the induced tree of coned-off spaces for $(\wt{M}_{l}, \mathcal{C})$ so that we may apply \cref{thm:combination}. We first need two lemmas.

\begin{lemma}
\label{lem:iteration_general}
Let $\omega\geqslant 0$ and let $m\geqslant k$ be an integer. There is a constant $M$ such that the following holds. 

Let $\gamma_1, \gamma_2\colon I \to R$ be immersed paths with $\len(\gamma_1), \len(\gamma_2)\leqslant \omega$ and let $\lambda\colon I\to R_{l}$ be an immersed path such that $\gamma_1*f^m(\lambda)*\gamma_2$ is path homotopic to an immersed path $\delta\colon I\to R_{l}$. If $\len(\delta)\geqslant M$, then $\delta = \delta_1*\alpha*\delta_2$ with $\len(\delta_1), \len(\delta_2)\leqslant M$ and with $\alpha$ a non-trivial path supported in $\bigsqcup_{i=1}^nR_{A_i}$.
\end{lemma}

\begin{proof}
Let $\Delta\to R$ be the graph obtained from $R_l\to R$ by replacing each edge in $R_l$ with the edge path obtained by applying $f^m$. As remarked in \cref{sec:folds}, there is a unique graph immersion $\Theta\to R$ that $\Delta\to R$ factors surjectively through. In particular, we have the following commutative diagram
\[
\begin{tikzcd}
\Delta \arrow[rd] \arrow[r] & \Theta \arrow[d] & {\Gamma(\psi^m(\pi_1(\Gamma, v))} \arrow[l, hook'] \arrow[ld] \\
                            & R                &                                                              
\end{tikzcd}
\]
where $\core(\Theta, v) = \Gamma(\psi^m(\pi_1(\Gamma, v))$. Finally denote by 
\[
\Gamma = \Theta\times_RR_{l}.
\]
By \cref{thm:acylindrical}, since $m\geqslant k$ we have that $\core(\Gamma)$ maps into $\bigsqcup_{i=1}^nR_{A_i}$. We set
\[
M = |E(\Gamma)| + 2\omega + 1.
\]
Let $\lambda'\colon I\to R$ be the immersed path with $\lambda'\sim f^m(\lambda)$. If $\len(\delta)\geqslant M$, then $\gamma_1*\lambda'*\gamma_2$ is path homotopic to the immersed path
\[
\delta = \gamma_1'*\lambda''*\gamma_2'
\]
where $\gamma_1'$ is a prefix of $\gamma_1$, $\lambda''$ is a subpath of $\lambda'$ of length at least $|E(\Gamma)|+1$ and where $\gamma_2'$ is a suffix of $\gamma_2$.

The path $\lambda'$ lifts to $\Theta$ by assumption. In particular, the subpath $\lambda''$ also lifts to $\Theta$. Since $\lambda''$ is supported in $R_{l}$, we see that $\lambda''$ also lifts to the pullback $\Gamma$. Since the core of each component of $\Gamma$ maps to $\bigsqcup_{i=1}^nR_{A_i}$ we see that 
\[
\lambda'' = \lambda_1''*\alpha*\lambda_2''
\]
with $\len(\lambda_1'') + \len(\lambda_2'') \leqslant |E(\Gamma)|$ and with $\alpha$ supported in $\bigsqcup_{i=1}^nR_{A_i}$.
\end{proof}

\begin{lemma}
\label{lem:iteration}
Let $m\geqslant k$ be an integer and let $\lambda\colon I\to R_l$ be a path connecting $v_i$ with $v_j$ for some $1\leqslant i, j\leqslant n$. If $f^m(\lambda)$ is path homotopic to an immersed path $\delta\colon I \to R_l$, then $\sigma^m(i) = \sigma^m(j)$ and $\delta$ is supported in $R_{A_{\sigma^m(i)}}$. 
\end{lemma}

\begin{proof}
Let $M$ be the constant from \cref{lem:iteration_general} when applied to $\omega = 0$ and $m$. Since $f^m(R_{A_i}), f^m(R_{A_j})\subset R_l$ by construction, we see that there are paths $a_i\colon I\to R_{A_i}$, $a_j\colon I\to R_{A_j}$ so that, if $b_i, b_j$ denote the immersed paths with $b_i\sim f^m(a_i), b_j\sim f^m(a_j)$, then $\len(b_i), \len(b_j)\geqslant M$ and $b_i*\delta*b_j$ is an immersed path supported in $R_l$. By definition of $M$, we see that $\delta$ is supported in $\bigsqcup_{i=1}^nR_{A_i}$.
\end{proof}

To avoid double superscripts, we shall denote by $\overline{M}_l$ the induced tree of coned-off spaces.

\begin{proposition}
\label{prop:hallways}
If $m\geqslant k$ and $\rho\geqslant 0$, then there is a constant $L\geqslant 0$ such that the following holds. If $h\colon [-m, m]\times I\to \overline{M}_{l}$ is an essential hallway, then
\begin{itemize}
\item $h$ is not cone-bounded,
\item the girth of $h$ is at most $L$ if $h$ is $\rho$-thin.
\end{itemize}
\end{proposition}

\begin{proof}
After performing a homotopy to $h$, we may assume that the path $h\mid[-m, m]\times \{0\}$ is of the form
\[
t_{-m}*g_{-m+1}*\ldots*g_0*t_0*g_1*t_1*\ldots*g_m*t_m
\]
where each $g_i$ is supported in a copy of $\wt{R}_{l}$, where $t_{-m}$ is a half edge of the form $[0, 1]\times\{x_{-m}\}$, $t_m$ is a half edge of the form $[-1, 0]\times\{x_m\}$ and where $t_i$ is an edge of the form $[-1, 1]\times\{x_i\}$ for all other values of $i$ and where here $x_i\in X_{e_i}^{(0)}$ is a 0-cell in an edge space associated with the edge $\pi(t_i) = e_i\in E(T)$. Note that this homotopy only increases lengths of the paths $h\mid[i, i+1]\times\{0\}$ by a constant $\xi$ depending on the metrics on the 2-cells in $\overline{M}_l$ (of which there are finitely many types). Similarly, we may assume that $h\mid[0, m]\times \{1\}$ is a path of the form
\[
t'_{-m}*g'_{-m+1}*\ldots*g'_0*t'_0*g'_1*t'_1*\ldots*g'_m*t'_m.
\]
If $h$ was $\rho$-thin before the homotopy, $h$ will by $\xi\rho$-thin after the homotopy. For ease of notation, replace $\rho$ with $\xi\rho$.

For each $i$, denote by $h_i^{\pm}$ the geodesic in $X_{e_i^{\pm}}$ connecting the endpoints of the path $\partial^{\pm}_{e_i}\circ h\mid\{i\}\times I$.

Suppose first for a contradiction that $h$ is cone-bounded. Then each $g_i, g_i'$ is trivial by definition and so $f^m(\mathfrak{p}\circ h_0^+)$ is path homotopic to $\mathfrak{p}\circ h_m^+$. Now \cref{lem:iteration} implies that $h_m^+$ is supported in a copy of the coning off of the universal cover of $R_{A_i}$ for some $i$. This implies that $h\mid\{m\}\times I$ is a geodesic connecting a cone point with itself which is not possible (since it would have to be trivial), a contradiction. Thus, $h$ cannot be cone-bounded as claimed.

Now suppose that $h$ is $\rho$-thin. Then each $g_{i}$, $g_{i}'$ has length at most $\rho$. By definition of $K, C$, we have that 
\begin{equation}
\label{eq: length bound 1}
K^{-1}\cdot \len(h\mid\{i\}\times I)  - C\leqslant \len(h_i^{\pm})\leqslant K\cdot \len(h\mid\{i\}\times I) + C
\end{equation}
for each $i$. The map $h|[i, i+1]\times I$ implies that 
\begin{equation}
\label{eq: path_homotopy}
g_{i+1}*h_i^+*\overline{g}_{i+1}'\sim h_{i+1}^-.
\end{equation}
where $\sim$ denotes path homotopy within the corresponding vertex space. We have
\begin{align*}
\len(h_i^+) - 2\rho \leqslant \len(h_i^+) - \len(g_i) - \len(g_i') &\leqslant \len(h_{i+1}^-)\\
						&\leqslant \len(h_i^+) + \len(g_i) + \len(g_i') \leqslant \len(h_i^+) + 2\rho. 
\end{align*}
Combining this with \eqref{eq: length bound 1} we see that
\begin{align*}
\len(h\mid\{i\}\times I)&\leqslant K\len(h^{\pm}_i) + KC\\
					&\leqslant K(\len(h^{\mp}_{i\pm1}) + 2\rho) + KC\\
					&\leqslant K(K\len(h\mid\{i\pm1\}\times I) + C + 2\rho) + KC\\
					&\leqslant K^2\len(h\mid\{i\pm1\}\times I) + 2K(C + \rho).
\end{align*}
Thus, by induction on $-m\leqslant n\leqslant m$, we have
\begin{align}
\label{eq: length bound 2}
K^{-2|n|}\cdot \len(h\mid\{0\}\times I) - 2(C+\rho)&\leqslant \len(h\mid\{n\}\times I)\\
\label{eq: length bound 3}				&\leqslant K^{2|n|}\cdot(\len(h\mid\{0\}\times I) + 2(C+\rho)).
\end{align}
Now suppose that for some $-m\leqslant i< m$, the paths $\overline{t}_i$ and $t_{i+1}$ follow the induced orientation on $T$. Then $\partial_{e_i}^+(X_{e_i})$ and $\partial_{e_{i+1}}^-(X_{e_{i+1}})$ both project to $R_{l-1}$ under the cover $c$. Since $R_{l-1}$ is a subgraph of $R_{l}$, we see that either $\partial_{e_i}^+(X_{e_i}) = \partial_{e_{i+1}}^-(X_{e_{i+1}})$ and $e_i = \overline{e}_{i+1}$ or $\partial_{e_i}^+(X_{e_i})\cap\partial_{e_{i+1}}^-(X_{e_{i+1}}) = \emptyset$ and $e_i\neq \overline{e}_{i+1}$. But since $h$ is an essential hallway, we must be in the latter case. The paths $g_{i+1}$ and $g_{i+1}'$ thus connect the two disjoint subtrees $\partial_{e_i}^+(X_{e_i})$ and $\partial_{e_{i+1}}^-(X_{e_{i+1}})$. Hence $g_{i+1}, g_{i+1}'$ exit $\partial_{e_i}^+(X_{e_i})$ at the same point and enter $\partial_{e_{i+1}}^-(X_{e_{i+1}})$ at the same point. This implies that $\len(h_{i+1}^-), \len(h_i^+)\leqslant 2\rho$. Combining with \eqref{eq: length bound 1}, we see that $\len(h\mid \{i\}\times I)\leqslant K(2\rho +C)$. By \eqref{eq: length bound 2} we see that
\[
\len(h\mid \{0\}\times I)\leqslant 2(K^{2m}\rho + C + \rho)
\]

Now assume that no such $i$ exists. In particular, there is at most one $-m\leqslant i< m$ so that $t_i$ and $\overline{t}_{i+1}$ follow the induced orientation on $T$ and for all other $i$, either $t_i, t_{i+1}$ or $\overline{t}_i, \overline{t}_{i+1}$ follow the induced orientation on $T$. In any case, there is an $i$ so that (after possibly flipping $h$), the edges $t_i, t_{i+1}, \ldots, t_{i+m}$ all follow the induced orientation on $T$.

The null-homotopy $h\mid[i, i+m]\times I$ implies that
\[
f^m(\mathfrak{p}\circ h_i^+) \sim g*(\mathfrak{p}\circ h_{i+m}^+)*\overline{g}'
\]
in $R$, where here
\begin{align*}
g &\sim f^m(\mathfrak{p}\circ g_{i})*f^{m-1}(\mathfrak{p}\circ g_{i+1})*\ldots *(\mathfrak{p}\circ g_{i+m})\\
g' &\sim f^m(\mathfrak{p}\circ g_i')*f^{m-1}(\mathfrak{p}\circ g_{i+1}')*\ldots *(\mathfrak{p}\circ g_{i+m}')
\end{align*}
are paths (of shortest length in their homotopy class) in $R$. We have that
\[
\len(g), \len(g') \leqslant mK^m(\rho+C).
\]
Since $\mathfrak{p}\circ h_i^+$ and $\mathfrak{p}\circ h_{i+m}^+$ are immersed paths in $R_{l}$, we may apply \cref{lem:iteration_general} with $\omega = mK^m(\rho+C)$ to conclude that 
\[
\len(h_{i+m}^+)\leqslant 2M + 1.
\]
Here we are using the fact that a geodesic in the coned-off tree of spaces connecting two vertices in a copy of $\widetilde{R}_{A_i}$ (for $1\leqslant i\leqslant n$) has length at most one. By \eqref{eq: length bound 3}, we see that
\[
\len(h\mid \{0\}\times I)\leqslant K^{2m}(2M+1 + 2(C+\rho)). 
\]
This completes the proof.
\end{proof}

\subsection{Proof of \cref{thm:rel_hyp}}

We showed that $\widetilde{M}_{l}$ is a tree of relatively hyperbolic spaces satisfying the qi-embedded condition, the strictly type-preserving condition and the qi-preserving electrocution condition in \cref{lem:conditions}. \cref{prop:hallways} implies that the induced tree of coned-off spaces $\overline{M}_l$ satisfies the hallways flare and the cone-bounded hallways strictly flare condition. By \cref{thm:combination} this implies that $\widetilde{M}_{l}$ is hyperbolic relative to the family of maximal cone-subtrees. Hence, $\pi_1(M_{l}, v)$ is hyperbolic relative to the subgroups $\{\pi_1(M[i]\cup r_i, v_i)\}_{[i]}$ by \cref{thm:group_combination} and so $(M(\psi), \mathcal{P})$ is relatively hyperbolic by \cref{lem:submapping_torus}.

\section{Splittings of subgroups induced by a graph pair and local relative quasi-convexity}
\label{sec:locally_rel_qc}

In this section we prove the final part of our main theorem.

\begin{theorem}
\label{thm:locally_rel_qc}
A finitely generated mapping torus $M(\psi)$ of a free group monomorphism $\psi\colon \F\to \F$ is locally relatively quasi-convex with respect to the relatively hyperbolic structure from \cref{thm:rel_hyp}.
\end{theorem}

In order to prove \cref{thm:locally_rel_qc} we shall require several auxiliary results on graph pairs and, in particular, certain direct limits of graph pairs. Thus, for the next sections we shall assume the notation and set-up from \cref{sec:graph_pairs}.

\subsection{Induced splittings from graph pairs}

We begin by relating certain maps of graph pairs to induced splittings of subgroups. First we should explain what exactly we mean by an induced splitting. If $G$ is a group acting on a tree $T$ without edge inversions, then there is a natural graph of groups structure $\mathcal{G}$ that can be put on the quotient graph $G\backslash T$ so that $\pi_1(\mathcal{G}) \isom G$. If $H$ is a subgroup of $G$ and acts on a subtree $S\subset T$, then there is also a quotient graph of groups $\mathcal{H}$, with underlying graph $H\backslash S$, so that $\pi_1(\mathcal{H}) \isom H$ and a natural morphism of graphs of groups $\gamma_H\colon \mathcal{H}\to\mathcal{G}$ so that $(\gamma_H)_*$ induces the inclusion $H\leqslant G$. This is all explained in detail in \cite{Se80,Ba93}. The graph of groups $\mathcal{H}$ along with the morphism $\gamma_H$ is the \emph{induced splitting} of $H$.

\begin{proposition}
\label{prop:intersection}
Let $(Z_1, X_1)$ be a finite tight minimal $\psi$-invariant graph pair for $G$ and let $\rho\colon (Z_2, X_2)\to (Z_1, X_1)$ be a map of graph pairs with $(Z_2, X_2)$ a tight $\psi$-invariant graph pair for $H$ such that
\begin{align*}
X_2^{\#} &= Z_2^{\#}\cap X_1^{\#}\\
\psi(X_2^{\#}) &= Z_2^{\#}\cap \psi(X_1^{\#}).
\end{align*}
Then
\[
H\isom \langle Z_2^{\#}, t \mid t^{-1}xt = \psi(x), \forall x\in X_2^{\#}\rangle
\]
is a HNN-splitting of $H$ induced by the HNN-splitting
\[
G \isom \langle Z_1^{\#}, t \mid t^{-1}xt = \psi(x), \forall x\in X_1^{\#}\rangle
\]
of $G$. In particular, we have
\begin{align*}
Z_2^{\#} &= H\cap Z_1^{\#}\\
X_2^{\#} &= H\cap X_1^{\#}.
\end{align*}
\end{proposition}

\begin{proof}
By \cref{FHmain}, we have that $\langle Z_1^{\#}, t\rangle \isom Z_1^{\#}*_{\phi_1}$ where $\phi_1 = \psi\mid X_1^{\#}$. By assumption, we have that $X_2^{\#} = Z_2^{\#}\cap X_1^{\#}$ and $\psi(X_2^{\#}) = Z_2^{\#}\cap \psi(X_1^{\#})$. Let $\phi_2 = \psi\mid X_2^{\#}$ and consider the homomorphism $Z_2^{\#}*_{\phi_2} \to Z_1^{\#}*_{\phi_1}$ given by $\rho_*$ on $\pi_1(Z_2, v_Z)\cong Z_2^{\#}$ and given by the identity on $t$. In terms of graphs of groups, this homomorphism is induced by an immersion of graphs of groups $\lambda_H\colon\mathcal{H}\to \mathcal{G}$ in the sense of Bass \cite{Ba93}, where here $\mathcal{H}, \mathcal{G}$ are the graphs of groups corresponding to the HNN-extensions $Z_2^{\#}*_{\phi_2}, Z_1^{\#}*_{\phi_1}$ respectively. In particular, the homomorphism $(\lambda_H)_*$ is injective and $\mathcal{H}$ is the induced splitting of $(\lambda_H)_*(\pi_1(\mathcal{H}))$ by \cite[Proposition 2.7]{Ba93}. This can also be seen directly by looking at the HNN-extension normal forms. Since $(\lambda_H)_*$ factors surjectively through the inclusion $H\to Z_1^{\#}*_{\phi_1}$, we see that $H$ has the claimed splitting and presentation. The fact that $Z_2^{\#} = H\cap Z_1^{\#}$ and $X_2^{\#} = H\cap X_1^{\#}$ follows from the fact that $\mathcal{H}$ is the induced splitting for $H$ (or by looking at the normal forms).
\end{proof}

\subsection{Direct limits of graph pairs}
\label{sec:direct_limit}

In this section we shall construct the direct limits of graph pairs we need for the proof of \cref{thm:locally_rel_qc} and shall prove they have some useful properties.

Recall that if $\{X_i\}_{i\geqslant 0}$ is a collection of graphs and $\{f_{ij}\colon X_i\to X_{j}\}_{i< j\in \N}$ is a collection of graph maps so that $f_{jk}\circ f_{ij} = f_{ik}$ for all $i< j< k$, then the direct limit can be described explicitly as the graph
\[
\check{X} = \lim_{i\to\infty} X_i
\]
with
\begin{align*}
V(\check{X}) &= \bigsqcup_{i\in\N}V(X_i) /\sim \quad \text{where } V(X_i) \ni v\sim w\in V(X_j) \text{ if } f_{ij}(v) = w\\
E(\check{X}) &= \bigsqcup_{i\in\N}E(X_i) /\sim \quad \text{where } E(X_i) \ni e\sim f\in E(X_j) \text{ if } f_{ij}(e) = f
\end{align*}
together with the collection of maps
\[
\{\check{f}_i\colon X_i\to \check{X}\}_{i\geqslant 0}
\]
given by $\check{f}_i(v) = [v]$ and $\check{f}_i(e) = [e]$ for all $v\in V(X_i), e\in E(X_i)$. The direct limit satisfies the following universal property: if $\{g_i\colon X_i\to Y\}_{i\in\N}$ are graph maps so that $g_j\circ f_{ij} = g_i$ for all $i<j$, then there is a canonical map $\check{g}\colon \check{X}\to Y$ so that $g_i = \check{g}\circ\check{f}_i$ for all $i\geqslant 0$.

\begin{lemma}
\label{lem:direct_limit_pair}
Let $(Z, X)$ be a $\psi$-invariant graph pair for $H$. If
\[
(Z, X) = (Z_0, X_0) \to (Z_1, X_1) \to\ldots \to (Z_k, X_k)\to \ldots
\]
is a sequence of maps of $\psi$-invariant graph pairs for $H$ with
\[
\rr(Z_{i+1}, X_{i+1}) \leqslant \rr(Z_i, X_i)
\]
for each $i\geqslant 0$, then
\[
(\check{Z}, \check{X}) = \left(\lim_{i\to \infty}Z_i, \lim_{i\to \infty}X_i\right),
\]
along with the induced map $f_{\check{Z}}\colon(\check{Z}, \check{X})\to (R, R)$, is a $\psi$-invariant graph pair for $H$ with
\[
\rr(\check{Z}, \check{X}) \leqslant \rr(Z, X).
\]
\end{lemma}

\begin{proof}
Since each $X_i, Z_i$ is connected and contains the basepoint, $\check{X}, \check{Z}$ are also connected and contain the basepoint. Thus $(\check{Z}, \check{X})$ is a graph pair. Since each graph pair in the sequence is $\psi$-invariant for $H$, we have that $(\check{Z}, \check{X})$ is also a $\psi$-invariant graph pair for $H$. 

If $\rr(Z, X)<\infty$ and $\rr(\check{Z}, \check{X})> \rr(Z, X)$, then there would be a finite connected subgraph $A\subset \check{Z}$ such that $A\cap \check{X}$ is a tree and such that $\rr(Z, X)<\rr(A, A\cap \check{X})$. But then there would be some $i$ such that $Z_i$ contains a subgraph $B$ such that the graph pair $(B, B\cap X_i)$ maps isomorphically to $(A, A\cap \check{X})$. But this implies that $\rr(Z, X)<\rr(Z_i, X_i)$, a contradiction. Thus, if $\rr(Z, X)<\infty$, then $\rr(\check{Z}, \check{X})\leqslant \rr(Z, X)$. If $\rr(Z, X) = \infty$, then certainly $\rr(\check{Z}, \check{X})\leqslant \rr(Z, X)$.
\end{proof}

\begin{lemma}
\label{lem:tight}
If $(Z, X)$ is a $\psi$-invariant graph pair for $H$, then there exists a tight $\psi$-invariant graph pair $(\check{Z}, \check{X})$ for $H$ such that 
\begin{align*}
\rr(\check{Z}, \check{X}) &\leqslant \rr(Z, X).
\end{align*}
If $f_Z$ factors through a tight $\psi$-invariant graph pair $f_{Z'}\colon (Z', X')\to (R, R)$, then $f_{\check{Z}}$ also factors through $f_{Z'}$.
\end{lemma}

\begin{proof}
We define a sequence of maps of graph pairs
\[
(Z, X) = (Z_0, X_0) \to (Z_1, X_1)\to \ldots \to (Z_k, X_k)\to \ldots
\]
so that $(Z_{i+1}, X_{i+1})$ is obtained from $(Z_i, X_i)$ by folding and adding a loop if necessary, ensuring that for each $i$, each fold that can be performed is eventually performed in the sequence. In this way, the direct limit 
\[
f_{\check{Z}} = \lim_{i\to \infty}f_i\colon \check{Z} = \lim_{i\to \infty}Z_i \to R
\]
is an immersion. Combining \cref{lem:graph_pair_faalin} with \cref{lem:direct_limit_pair}, we see that $(\check{Z}, \check{X})$ is a $\psi$-invariant graph pair for $H$ with $\rr(\check{Z}, \check{X}) \leqslant \rr(Z, X)$. Since each $f_{Z_i}$ factors through $f_{Z'}$ by \cref{lem:graph_pair_faalin}, we also have that $f_{\check{Z}}$ factors through $f_{Z'}$ by the universal property of direct limits.
\end{proof}

Now we may construct maps of graph pairs like those in \cref{prop:intersection} in terms of direct limits of maps of finite graph pairs.

\begin{proposition}
\label{prop:induced_pair}
If $(Z_2, X_2)$ is a $\psi$-invariant graph pair for $H$ and $\rho\colon (Z_2, X_2)\to (Z_1, X_1)$ a map of $\psi$-invariant graph pairs with $(Z_1, X_1)$ tight, then there is a commutative diagram
\[
\begin{tikzcd}
{(Z_2, X_2)} \arrow[d] \arrow[r, "\rho"]                & {(Z_1, X_1)} \\
{(\check{Z}_2, \check{X}_2)} \arrow[ru, "\check{\rho}"'] &             
\end{tikzcd}
\]
where $(\check{Z}_2, \check{X}_2)$ is a tight $\psi$-invariant graph pair for $H$ and such that
\begin{align*}
\check{X}_2^{\#} &= \check{Z}_2^{\#}\cap X_1^{\#}\\
\psi(\check{X}_2^{\#}) &= \check{Z}_2^{\#}\cap \psi(X_1^{\#})\\
\rr(\check{Z}_2, \check{X}_2)&\leqslant \rr(Z_2, X_2).
\end{align*}
\end{proposition}

\begin{proof}
Consider the following sequence:
\[
(Z_2, X_2) = (Z_0', X_0') \to (Z_1', X_1')\to \ldots \to (Z_1, X_1)
\]
of graph pairs constructed as follows. 

The graph pair $(Z_{2i+1}', X_{2i+1}')$ is the tight $\psi$-invariant graph pair obtained from $(Z_{2i}', X_{2i}')$ as in \cref{lem:tight}. Note that $\rr(Z'_{2i+1}, X'_{2i+1})\leqslant \rr(Z'_{2i}, X'_{2i})$. 

The graph pair $(Z_{2i}', X_{2i}')$ is the $\psi$-invariant graph pair obtained from the tight $\psi$-invariant graph pair $(Z_{2i-1}', X_{2i-1}')$ as follows. Since $X_1^{\#}$ is a subgraph of $Z_1^{\#}$, we have that $X_{2i-1}'' = Z_{2i-1}'\times_{Z_1}X_1$ is a subgraph of $Z_{2i-1}'$. Moreover, we have $X_{2i-1}''^{\#} = Z_{2i-1}'^{\#}\cap X_1^{\#}$ by \cref{lem:double_coset}. Since $X_{2i-1}'\subset X_{2i-1}''$, we have that $X_{2i-1}''^{\#} = X_{2i-1}'^{\#}*K_{2i}$ for some $K_{2i}$. Now we let:
\begin{align*}
X_{2i}' &= \Gamma(\psi^{-1}(Z_{2i-1}'^{\#}\cap \psi(X_1^{\#})))\vee X_{2i-1}''\\
Z_{2i}' &= \Gamma(\psi^{-1}(Z_{2i-1}'^{\#}\cap \psi(X_1^{\#})))\vee Z_{2i-1}'\vee\Gamma(\psi(K_{2i}))
\end{align*}
By construction, we have that $(Z_{2i}', X_{2i}')$ is a $\psi$-invariant graph pair for $H$ and $\rr(Z_{2i}', X_{2i}') = \rr(Z_{2i-1}', X_{2i-1}')$. 

Applying \cref{lem:direct_limit_pair} to this sequence of maps of graph pairs we obtain a $\psi$-invariant graph pair $(\check{Z}_2, \check{X}_2)$ for $H$, with $\rr(\check{Z}_2, \check{X}_2)\leqslant \rr(Z_2, X_2)$, and a map $\check{\rho}\colon (\check{Z}_2, \check{X}_2)\to (Z_1, X_1)$ of graph pairs. 

The pair $(\check{Z}_2, \check{X}_2)$ is tight since it is also a direct limit of tight graph pairs (consider the subsequence with odd indices). 

Now let $\lambda\colon I\to \check{Z}_2$ be a loop representing an element in $X_1^{\#}$. For sufficiently large $i$, the loop $\lambda$ lifts to a loop in $(Z_{2i-1}', X_{2i-1}')$. By construction, for all $j\geqslant 2i+1$, the image of this loop in $Z_j'$ lies in $X_j'$. Hence $\lambda$ itself lifts to $\check{X}_2$. By a similar argument, any loop representing an element in $\psi(X_1^{\#})$ represents an element in $\psi(\check{X}^{\#}_2)$. Hence we have
\begin{align*}
\check{X}_2^{\#} &= \check{Z}_2^{\#}\cap X_1^{\#}\\
\psi(\check{X}_2^{\#}) &= \check{Z}_2^{\#}\cap \psi(X_1^{\#})
\end{align*}
as required.
\end{proof}

\subsection{Lifting graph pair maps}

The aim of this section is to prove \cref{prop:preimage_cut} below, the proof of which will hinge on a property of lifts to graph pairs.

\begin{proposition}
\label{prop:preimage_cut}
Let $\rho\colon (Z_2, X_2)\to (Z_1, X_1)$ be a map of tight $\psi$-invariant graph pairs so that $Z_1$ is finite and so that $\psi(X^{\#}_1)\cap Z^{\#}_2= \psi(X^{\#}_2)$. There is a constant $\kappa$ such that for any finite collection of points $\mathcal{P}\subset Z_2$, there is a collection of at most $\kappa|\mathcal{P}|$ points $\mathcal{Q}\subset X_2$ with the following property. 

If $(U, v_U)\to (X_2, v_{Z_2})$ is a pointed map restricting to a map $\core(U)\to X_2 - \mathcal{Q}$, then the pointed map $\Gamma(\psi(U^{\#}))\to Z_2$ restricts to a map $\Gamma[\psi(U^{\#})]\to Z_2 - \mathcal{P}$.
\end{proposition}

For the rest of this section we fix a map
\[
\rho = (\rho_{Z_2}, \rho_{X_2}) \colon (Z_2, X_2) \to (Z_1, X_1)
\]
of $\psi$-invariant tight graph pairs and a cellular (not necessarily combinatorial) map
\[
f_1\colon X_1\to Z_1
\]
so that $(f_1)_* = \psi\mid X_1^{\#}$. For simplicity, we shall assume that $\rho_{Z_2}\colon Z_2\to Z_1$ is actually a covering. We shall also always be assuming that our spaces and maps are pointed, but suppress the basepoint from the notation.

As explained in \cref{sec:folds} (we pass to a direct limit of folds when $X_1$ is infinite), we may decompose $f_1$ as a composition of two maps $f_1 = h_1\circ g_1$ where $g_1\colon X_1\to X_1^f$ factors as a sequence of folds (with respect to an appropriate subdivision of the edges in $X_1$) and where $h_1\colon X_1^f\to Z$ is a combinatorial immersion.
\[
\begin{tikzcd}
Z_1                     & X_1 \arrow[r, "g_1"] \arrow[l, phantom, sloped, "\supseteq"]                     & X_1^f \arrow[ll, "h_1"', bend left=49]            
\end{tikzcd}
\]
Now consider the cellular map $g_1\circ\rho_{X_2}\colon X_2\to X_1^f$. Again, this factors as a sequence of folds $g_2\colon X_2\to X_2^f$ followed by an immersion $\rho_{X_2^f}\colon X_2^f\to X_1^f$. 

Since $Z_2\to Z_1$ is a covering, the map $X_2^f\to Z_1$ lifts uniquely to a (pointed) map to $Z_2$. Denote by $h_2\colon X_2^f\to Z_2$ this lift. We may summarise our maps in the following commutative diagram:
\begin{equation}
\label{diagram}
\begin{tikzcd}
X_2 \arrow[r, "g_2"] \arrow[d, "\rho_{X_2}"] \arrow[rr, "f_2", bend left=49] & X_2^f \arrow[d, "\rho_{X_2^f}"] \arrow[r, "h_2"] & Z_2 \arrow[d, "\rho_{Z_2}"] \\
X_1 \arrow[r, "g_1"] \arrow[rr, "f_1", bend right=49]                        & X_1^f \arrow[r, "h_1"]                           & Z_1                        
\end{tikzcd}
\end{equation}
Call $f_2$ the \emph{lift} of $f_1$ to $(Z_2, X_2)$. Note that we have
\begin{equation}
\label{eq:core}
\core\left(X_i^f, v_{X_i^f}\right) = \Gamma(\psi(X^{\#}_i))
\end{equation}
for $i = 1, 2$.

\begin{lemma}
\label{lem:preimage_bound1}
If $\psi(X^{\#}_1)\cap Z^{\#}_2= \psi(X^{\#}_2)$, then for each $x\in Z_2$, the set $h_2^{-1}(x)$ injects into the set $h_1^{-1}\left(\rho_{Z_2}(x)\right)$ via $\rho_{X_2^f}$.
\end{lemma}

\begin{proof}
Let $x_1, x_2\subset X_2^f$ be two points such that $h_2(x_1) = h_2(x_2) = x$. Suppose that $\rho_{X_2^f}(x_1) = \rho_{X_2^f}(x_2)$. Let $p_1, p_2\colon I\to X_2^f$ be paths connecting the basepoint with $x_1$ and $x_2$ respectively. Note that if $x_1\neq x_2$, then $h_2\circ p_1$ is not path homotopic to $h_2\circ p_2$ since $h_2$ is an immersion. Then if $x_1\neq x_2$, we have that $(h_2\circ p_1) * (h_2\circ \overline{p}_2)$ is a loop in $Z_2$ that is not null-homotopic. Moreover, if $x_1\neq x_2$, then $(\rho_{X_2^f}\circ p_1)*(\rho_{X_2^f}\circ \overline{p}_2)$ is a loop in $X_1^f$ that is not null-homotopic. By \eqref{eq:core} this implies that if $x_1\neq x_2$ there is a non-trivial element, $(h_1)_*([(\rho_{X_2^f}\circ p_1)*(\rho_{X_2^f}\circ \overline{p}_2)])$, that lies in $\psi(X^{\#}_1)\cap Z^{\#}_2$ that does not lie in $\psi(X^{\#}_2)$. Since this is not possible, this implies that $x_1 = x_2$. Now the result follows from the commutativity of \eqref{diagram}.
\end{proof}

\begin{lemma}
\label{lem:preimage_bound2}
For each $x\in X_2^f$, the set $g_2^{-1}(x)$ injects into the set $g_1^{-1}\left(\rho_{X_2^f}(x)\right)$ via $\rho_{X_2}$.
\end{lemma}

\begin{proof}
Let $x_1, x_2\in g_2^{-1}(x)$ and suppose that $\rho_{X_2}(x_1) = \rho_{X_2}(x_2)$. Let $p_1, p_2\colon I\to X_2$ be immersed paths connecting the basepoint with $x_1, x_2$ respectively. If $x_1 \neq x_2$, then since $\rho_{X_2}$ is an immersion, the loop $(\rho_{X_2}\circ p_1)*(\rho_{X_2}\circ\overline{p}_2)$ is not nullhomotopic. By commutativity of \eqref{diagram} we see that also $(g_2\circ p_1)*(g_2\circ\overline{p}_2)$ is not nullhomotopic if $x_1\neq x_2$. But since $(X_2^f)^{\#} = \psi(X_2^{\#})$, it follows that the loop $(\rho_{X_2}\circ p_1)*(\rho_{X_2}\circ\overline{p}_2)$ must lift to a loop in $X_2$. Thus, we conclude that $x_1 = x_2$. This implies the result.
\end{proof}

Combining \cref{lem:preimage_bound1,lem:preimage_bound2} we obtain:

\begin{corollary}
\label{cor:preimage_bound}
If $\psi(X^{\#}_1)\cap Z^{\#}_2= \psi(X^{\#}_2)$, then for each $x\in Z_2$, the set $f_2^{-1}(x)$ injects into $f_1^{-1}\left(\rho_{Z_2}(x)\right)$ via $\rho_{X_2}$. 

In particular, if $Z_1$ is finite, then there is a constant $\kappa$ so that $|f_2^{-1}(x)|\leqslant \kappa$ for all $x\in Z_2$.
\end{corollary}

\begin{remark}
\label{rem:preimage_remark}
When $f_1\colon X_1\to Z_1$ is a homotopy equivalence onto its image in $Z_1$, then $h_i\colon X_i^f\to Z_i$ is an inclusion for $i = 1, 2$. In this case the assumption that $\psi(X^{\#}_1)\cap Z^{\#}_2= \psi(X^{\#}_2)$ in \cref{cor:preimage_bound} can be dropped.
\end{remark}

\begin{proof}[Proof of \cref{prop:preimage_cut}]
Let $f_1\colon X_1\to Z_1$ be any cellular map so that $(f_1)_* = \psi\mid X_1^{\#}$. After possibly attaching some trees to $Z_2$ (so that $\rho_{Z_2}$ becomes a cover), a lift $f_2\colon X_2\to Z_2$ of $f_1$ exists. Let $\kappa$ be the constant from \cref{cor:preimage_bound}.

Now take 
\[
\mathcal{Q} = f_2^{-1}(\mathcal{P})\subset X_2
\]
and note that we have $|\mathcal{Q}|\leqslant \kappa|\mathcal{P}|$ by \cref{cor:preimage_bound}. By our choice of points $\mathcal{Q}$, we see that $f_2(X_2 - \mathcal{Q})\subset Z_2 - \mathcal{P}$ which implies the result.
\end{proof}

\subsection{Proof of \cref{thm:locally_rel_qc}}

We are going to recycle the set-up from \cref{sec:mapping_torus}. Recall that we have a graph $R$, an identification $\pi_1(R, v)\cong \F$ and a cellular map $f\colon R\to R$ so that $f_* = \psi$ and so that $G = M(\psi)\isom \pi_1(M(f))$. Recall also that for each $l\geqslant \mu$, we have that the compact partial mapping torus $M_l\subset M(f)$ (on the subgraph $R_{l}\subset R$) has
\[
G\cong \pi_1(R_{l}, v)*_{\phi_l} \cong \pi_1(M_l, v)
\]
where $\phi_l$ is the restriction of $\psi$ to $\pi_1(R_{l}, v)$.

By \cref{lem:conditions,prop:hallways}, the hypotheses of \cref{cor:bounded_hallways} are met for each HNN-extension decomposition $G \cong \pi_1(R_{l}, v)*_{\phi_l}$. Hence, by \cref{cor:bounded_hallways}, for all $l\geqslant \mu$ we have that $R_{l}^{\#}$ is relatively quasi-convex in $G$. Since $(R_{l}^{\#}, \{R_{A_i}^{\#}\}_{i=1}^n)$ (this is the induced relatively hyperbolic structure on $R_{l}^{\#}$) is locally relatively quasi-convex by \cref{lem:tree_rel_hyp} and since $\F = \bigcup_{i\geqslant\mu}R_{i}^{\#}$, we see that any finitely generated subgroup of $\F$ is relatively quasi-convex in $G$.

We now consider subgroups of $G$ that contain the element $t\in G$. The idea will be to apply \cref{prop:intersection} to the graph pair constructed in the following proposition, and then conclude relative quasi-convexity using \cref{thm:qc_criterion}.

\begin{proposition}
\label{prop:key_prop}
If $H\leqslant\F*_{\psi}$ is a finitely generated subgroup generated by a subgroup of $\F$ and $t$, then there is a tight $\psi$-invariant graph pair $(Z, X)\to (R, R)$ for $H$ with the following properties:
\begin{enumerate}
\item There is some $p\geqslant \mu$ so that $(Z, X)\to (R, R)$ factors as 
\[
(Z, X) \xrightarrow{\rho}(R_{p}, R_{p-1}) \injects (R, R)
\]
\item $\rr(Z, X)<\infty$.
\item We have
\begin{align*}
X_2^{\#} &= Z_2^{\#}\cap X_1^{\#}\\
\psi(X_2^{\#}) &= Z_2^{\#}\cap \psi(X_1^{\#}).
\end{align*}
\item $Z - \core(f_Z^{-1}(\bigsqcup_{i=1}^n(R_{A_i})))$ consists of finitely many 1-cells and 0-cells.
\end{enumerate}
\end{proposition}

\begin{proof}
Let $(Z', X')$ be a $\psi$-invariant graph pair for $H$. Since $H$ is finitely generated, $(Z', X')$ may be taken to be finite and so $\rr(Z', X')<\infty$. Let $p\geqslant \mu$ be the smallest integer so that $f_{Z'}(X')\subset R_{p-1}$. Then $f_{Z'}$ factors as a pair of maps of graph pairs:
\[
(Z', X') \xrightarrow{\rho'}(R_{p}, R_{p-1}) \injects (R, R).
\]
Using \cref{prop:induced_pair} we may obtain a tight $\psi$-invariant graph pair $(Z, X)$ for $H$ such that the map $f_Z$ factors as
\[
(Z, X) \xrightarrow{\rho}(R_{p}, R_{p-1}) \injects (R, R)
\]
We may also assume that $(Z, v_Z) = \core(Z, v_Z)$. Moreover, by \cref{prop:induced_pair}, we have that $\rr(Z, X)\leqslant \rr(Z', X')<\infty$ and that
\begin{align*}
X_2^{\#} &= Z_2^{\#}\cap X_1^{\#}\\
\psi(X_2^{\#}) &= Z_2^{\#}\cap \psi(X_1^{\#}).
\end{align*}

Now let $k$ be the constant from \cref{thm:acylindrical}. Since $\rr(Z, X)<\infty$, there is a finite set of points $\mathcal{P}\subset Z - X$ so that $Z - \mathcal{P}$ deformation retracts to $X$. Let $\mathcal{Q}\subset X$ be the points from \cref{prop:preimage_cut}. Applying \cref{prop:preimage_cut} repeatedly, we obtain a sequence of sets of points $\mathcal{P} = \mathcal{Q}_0, \mathcal{Q} = \mathcal{Q}_1, \mathcal{Q}_2, \ldots, \mathcal{Q}_k\subset Z$ so that if $U\subset Z - \bigcup_{i=0}^k\mathcal{Q}_i$ is a core connected subgraph, then $\Gamma[\psi^k(U^{\#})]\to R$ factors through $X$. By letting $k$ be the constant from \cref{thm:acylindrical}, we see that $f_Z(U)\subset R_{A_i}$ for some $i$. By \cref{prop:preimage_cut}, $|\bigcup_{i=0}^k\mathcal{Q}_i|<\infty$. In particular, since $(Z, v_Z) = \core(Z, v_Z)$, this implies that $Z - \core(Z -  \bigcup_{i=0}^k\mathcal{Q}_i))$ consists of finitely many 1-cells and 0-cells. Since $\core(Z - \bigcup_{i=0}^k\mathcal{Q}_i)\subset  \core(f_Z^{-1}(\bigsqcup_{i=1}^n(R_{A_i})))$, the same holds for $Z - \core(f_Z^{-1}(\bigsqcup_{i=1}^n(R_{A_i})))$ as required.
\end{proof}

Now let $H\leqslant G$ be a finitely generated subgroup, generated by a subgroup of $\F$ and by $t$. Then let $(Z, X)$ be the graph pair from \cref{prop:key_prop}. Since $Z - \core(f_Z^{-1}(\bigsqcup_{i=1}^nR_{A_i}))$ consists of finitely many 1-cells and 0-cells, we may use \cref{lem:tree_rel_hyp} to conclude that $Z^{\#}$ is relatively quasi-convex in $(R_{p}^{\#}, \{A_i^{\#}\}_{i=1}^n)$ (note that it may not be finitely generated). By \cref{prop:intersection}, we see that
\[
Z^{\#} = H\cap R_{p}^{\#}.
\]
Hence, we may apply \cref{thm:qc_criterion} to conclude that $H$ is relatively quasi-convex in $G$.

Finally, we now use arguments from \cite{FH99} to reduce the general case to the two cases we just handled. Let $H\leqslant G$ be any finitely generated subgroup and let $\phi\colon G\to \Z$ be the homomorphism given by quotienting $G$ by $\normal{\F}$. Each element of $G$ can be written as $t^i f t^{-j}$ for some $i, j\geqslant 0$ and $f\in \F$. By replacing $H$ with $H^{t^j}$ for $j\geqslant 0$ large enough, we may assume that $H$ is generated by a finite set of elements of the form $ft^{-j}$ for $j\geqslant 0$. We see that if $\phi(H) = 0$, then $H\leqslant \F$ and so is relatively quasi-convex in $G$ by the first part of this proof. If $\phi(H) = m\Z$ for some $m\geqslant 0$, then $H$ is generated by a finite subset of $\F$ and an element $ht^{-m}$. Since $G' = \langle \F, ht^{-m}\rangle$ is isomorphic to the mapping torus of $\gamma_h\circ\psi^m\colon\F\to\F$ (here $\gamma_h$ denotes conjugation by $h$), we see that $H$ is relatively quasi-convex in $G'$ by the second part of this proof. Since $G'$ has index $m$ in $G$, it is relatively quasi-convex in $G$ and so $H$ itself is relatively quasi-convex in $G$.

\subsection{Proof of \cref{main}}

Feighn--Handel's main theorem in \cite{FH99} states that every non-free finitely generated subgroup of the mapping torus $M(\psi)$ is isomorphic to a HNN-extension of a finitely generated free group with one of the associated subgroups a free factor. Combined with \cite[Proposition 2.1]{FH99}, this means that every finitely generated non-free subgroup of $M(\psi)$ is itself isomorphic to the mapping torus of a free group. The main theorem now follows by applying \cref{thm:acylindrical}, \cref{thm:rel_hyp} and \cref{thm:locally_rel_qc} to this mapping torus.

\subsection{Proof of \cref{thm:one-relator}}

Let $F$ be a free group, let $w\in F$ be an element and let $G = F/\normal{w}$ be the quotient one-relator group. If $\pi(w)\neq 2$, then $G$ is virtually free-by-cyclic by \cite{KL24b}. By a result of Louder--Wilton \cite[Lemma 6.10]{LW22}, $G$ does not contain any non-cyclic subgroups $H$ with $\chi(H) = 0$. Hence, $G$ contains no mapping tori of finitely generated non-trivial free groups by \cref{FHmain}. Now by \cref{cor:lqh} it follows that $G$ is locally quasi-convex hyperbolic.

Now suppose that every finitely generated subgroup of $G$ is quasi-convex. By \cite{Li25}, there is a sequence of finitely generated one-relator groups $G_N\leqslant \ldots\leqslant G_1\leqslant G_0 = G$ such that $G_N$ is finite cyclic (or trivial) and $G_i$ splits as a HNN-extension over $G_{i+1}$ (with finitely generated associated subgroups). Since each $G_i$ and each edge group for each HNN-extension is quasi-convex in $G$, this hierarchy is a quasi-convex hierarchy in the sense of Wise. Thus, by \cite{Wi21}, $G$ is virtually compact special. Thus, by \cite{KL24b} it is virtually free-by-cyclic. If $\pi(w) = 2$, then $G$ contains a torsion-free non-cyclic subgroup $H$ with $\chi(H) = 0$ by \cite{LW22}. Since $\chi$ is multiplicative with index, this would imply that a finite index subgroup of $H$ is free-by-cyclic with $\chi = 0$. Hence, a finite index subgroup of $H$ is \{finitely generated free\}-by-cyclic by \cref{FHmain}. Since finitely generated infinite index normal subgroups of hyperbolic groups are not quasi-convex, we see that $H$ is not locally quasi-convex. We reach a contradiction and conclude that $\pi(w)\neq 2$.

\section{Promoting properties from the maximal sub-mapping tori}
\label{sec:applications}

We now turn to further results on mapping tori of free groups which follow from known results for mapping tori of finitely generated free groups combined with \cref{main}. For this section, fix a free group $\F$ and a monomorphism $\psi\colon \F\to\F$ so that its mapping torus $G = M(\psi)$ is finitely generated.

\subsection{The Dehn function}

When $\F$ is finitely generated and $\psi$ is surjective, Bridson--Groves \cite{BG10} showed that $G$ has either linear or quadratic Dehn function. Mutanguha generalised this and showed that when $\F$ is finitely generated, $G$ has either linear, quadratic or exponential Dehn function \cite[Corollary 4.8]{Mu24}. When $\F$ is not finitely generated, by \cref{thm:rel_hyp} $G$ is hyperbolic relative to a finite collection of mapping tori of finitely generated free groups. When $\psi$ is surjective, then $G$ is hyperbolic relative to a finite collection of \{finitely generated free\}-by-$\Z$ subgroups by \cref{fg_kernel}. Combining the results of Mutanguha and Bridson--Groves with a result of Osin \cite[Corollary 2.41]{Os06}, we obtain the following corollary.

\begin{corollary}
\label{dehn_function}
A finitely generated mapping torus of a free group $M(\psi)$ has linear, quadratic or exponential Dehn function. If $\psi$ is surjective, then $M(\psi)$ has either linear or quadratic Dehn function.
\end{corollary}

\subsection{The conjugacy problem}

Bogopolski--Martino--Maslakova--Ventura \cite{BMMV06} showed that the conjugacy problem for \{finitely generated free\}-by-$\Z$ groups is decidable. Alan Logan then showed that mapping tori of finitely generated free groups have decidable conjugacy problem in \cite{Lo23}. Since Bumagin showed in \cite{Bu04} that a relatively hyperbolic group with peripheral subgroups with decidable conjugacy problem has decidable conjugacy problem, we obtain the following by \cref{main}.

\begin{corollary}
\label{conjugacy_problem}
A finitely generated mapping torus of a free group $M(\psi)$ has decidable conjugacy problem.
\end{corollary}

When the Dehn function of $M(\psi)$ is quadratic, the decidability of the conjugacy problem for $M(\psi)$ follows from a result of Ol'shanskii--Sapir \cite{OS06'}.

We remark that in general the conjugacy problem being decidable is not a property that passes to finite index subgroups or overgroups, see work of Collins--Miller \cite{CM77}. However, \cref{conjugacy_problem} shows that the decidability of the conjugacy problem passes to arbitrary finitely generated subgroups of mapping tori of free groups.

\subsection{The finitely generated intersection property}

A group $G$ has the \emph{finitely generated intersection property} (or \emph{f.g.i.p.}) if for any pair of finitely generated subgroups $H, K\leqslant G$, the intersection $H\cap K$ is also finitely generated. Bamberger--Wise characterised when a mapping torus of a finitely generated free group has the f.g.i.p. property in \cite{BW22}. Using this, we may also characterisation amongst all mapping tori of free groups.

\begin{theorem}
\label{fgip}
The following are equivalent for a finitely generated mapping torus of a free group $M(\psi)$:
\begin{enumerate}
\item\label{itm:fgip1} $M(\psi)$ has the f.g.i.p.
\item\label{itm:fgip2} $M(\psi)$ contains no subgroup isomorphic to a mapping torus of a finitely generated free group of rank $2$ or more.
\item\label{itm:fgip3} $\F$ contains no finitely generated free factor $H\leqslant \F$ of rank at least two so that $\psi^m(H)$ is conjugate into $H$ for some $m\geqslant 1$.
\end{enumerate}
\end{theorem}

\begin{proof}
Bamberger--Wise's result in \cite{BW22} states that a mapping torus of a finitely generated free group of rank at least two does not have the f.g.i.p. property. So now suppose that $G = M(\psi)$ contains no such subgroups. By \cref{thm:rel_hyp}, $G$ is hyperbolic relative to a finite (possibly empty) collection of subgroups isomorphic to $\bs(1, n)$ for various values of $n$. Here, $\bs(1, n)$ is the mapping torus of $\Z$ given by the homomorphism $i\mapsto ni$, known as the Baumslag--Solitar group. By \cref{thm:locally_rel_qc}, $G$ is locally relatively quasi-convex. By \cite[Theorem 1.2]{Hr10}, if $H, K\leqslant G$ are finitely generated subgroups, then $H\cap K$ is relatively quasi-convex in $G$ and hence is relatively hyperbolic with respect to the induced peripherals. Since the peripherals of $H\cap K$ are intersections of conjugates of peripherals for $H$ and for $K$ (which are all finitely generated) and since $\bs(1, n)$ has the f.g.i.p. by a result of Moldavanskii \cite{Mo68}, we see that $H\cap K$ is finitely generated. Hence, $G$ has the f.g.i.p. and we have established the equivalence between \eqref{itm:fgip1} and \eqref{itm:fgip2}. The equivalence between \eqref{itm:fgip2} and \eqref{itm:fgip3} follows from \cref{thm:acylindrical}.
\end{proof}

\subsection{The locally undistorted property}

If $G$ is a group with finite generating set $S$ and if $H\leqslant G$ is a subgroup with finite generating set $T\subset H$, then the \emph{distortion function} for $H$ in $G$ is defined as
\[
\delta_{H, T}^{G, S}(n) = \max\{ |h|_T \mid h\in H, |h|_S\leqslant n\}.  
\]
Up to a natural equivalence relation $\sim$, the distortion function does not depend on the chosen generating sets $S, T$. Denote by $\delta_{H}^G$ the $\sim$-equivalence class of distortion functions for $H\leqslant G$. A subgroup $H$ is \emph{undistorted} if $\delta_H^G(n)\sim n$, \emph{distorted} otherwise. The reader is directed to \cite{Fa94} for more information on distortion of subgroups. In this section we characterise which mapping tori of free groups have all their finitely generated subgroups undistorted. Although this has not been stated explicitly in the literature for mapping tori of finitely generated free groups, we show how this case actually follows from some known results.

\begin{theorem}
\label{thm:locally_undistorted}
The following are equivalent for a finitely generated mapping torus of a free group $M(\psi)$:
\begin{enumerate}
\item Every finitely generated subgroup of $M(\psi)$ is undistorted.
\item Every subgroup of $M(\psi)$ that is isomorphic to a mapping torus of a finitely generated free group, is virtually $F\times \Z$ for some free group $F$.
\item If $F\leqslant \F$ is a free factor, $f\in \F$ and $m\geqslant 1$ such that $f^{-1}\psi^m(F)f\leqslant F$, then the induced endomorphism $\gamma_f\circ\psi^m\colon F\to F$ is an isomorphism and has finite order in $\Out(F)$.
\end{enumerate}
\end{theorem}

We shall use the following facts about distortion without mention:
\begin{enumerate}
\item If $H\leqslant K\leqslant G$ are finitely generated groups and $H$ has finite index in $K$, then $\delta_H^G\sim\delta_K^G$.
\item If $H\leqslant K\leqslant G$ are finitely generated groups and $H$ is distorted in $K$, then either $H$ or $K$ is distorted in $G$.
\end{enumerate}

We begin by handling the $\F = F_n$ case.

\begin{lemma}
\label{lem:non_surjective_distortion}
If $\F = F_n$ and $\psi\colon F_n\to F_n$ is a non-surjective monomorphism, then the group $F_n$ is distorted in the mapping torus $M(\psi)$.
\end{lemma}

\begin{proof}
If $M(\psi)$ contains a subgroup isomorphic to $\bs(1, n)$ for $|n|\geqslant 2$, then $F_n$ contains an exponentially distorted infinite cyclic subgroup and so is itself distorted. Thus we may assume that it contains no such subgroup. By a result of Mutanguha \cite[Theorem 4.7]{Mu24}, $G = M(\psi)$ is hyperbolic relative to a (possibly empty) collection of (infinite index) \{finitely generated free\}-by-$\Z$ subgroups. If $t\in G$ is the element so that $t^{-1}ft = \psi(f)$ for $f\in F_n$, then we see that $\bigcap_{i=0}^{\infty}t^iF_nt^{-i} = F_n$. Since $\{t^iF_n\}_{i=0}^{\infty}$ is a collection of distinct cosets of $F_n$, it follows that the subgroup $F_n$ has infinite height in the sense of Hruska--Wise \cite{HW09}. Then by \cite[Theorem 1.4]{HW09}, $F_n$ is not relatively quasi-convex in $G$. Finally, by a result of Hruska \cite[Theorem 1.4]{Hr10}, $F_n$ is distorted in $G$.
\end{proof}

\begin{lemma}
\label{undistorted_F_times_Z}
Every finitely generated subgroup of $F_n\times\Z$ is undistorted.
\end{lemma}

\begin{proof}
Let $H\leqslant F_n\times\Z$ be a finitely generated subgroup. Then 
\[
H\cong (H/H\cap \Z)\times(H\cap \Z).
\]
Since $H/H\cap \Z$ is a finitely generated subgroup of $F_n$, it is undistorted in $F_n$. Thus, $H$ is undistorted in $F_n\times\Z$.
\end{proof}

\begin{proposition}
\label{prop:distortion_fg_case}
If $\F = F_n$ and $\psi\colon F_n\to F_n$ is an isomorphism so that every finitely generated subgroup of $M(\psi)$ is undistorted, then $\psi$ has finite order in $\Out(F_n)$.
\end{proposition}

\begin{proof}
Suppose first that $\psi$ has finite order. Then $G = M(\psi)$ has a finite index subgroup isomorphic to $F\times\Z$ for some free group $F$. Since every finitely generated subgroup of $F\times\Z$ is undistorted by \cref{undistorted_F_times_Z}, so is every finitely generated subgroup of $G$.

Now suppose that $\psi$ is polynomially growing of degree $d\geqslant 1$. Kudlinska proved in \cite[Theorem 3.4]{Ku24} that the group 
\[
H = \langle a, b, c, d \mid [a, b], [b, c], [c, d]\rangle
\]
is a subgroup of $G$. This is a right-angled Artin group on the line graph with four vertices and three edges. In particular, a result of Tran \cite[Theorem 1.1]{Tr17} shows that the kernel of the map to $\Z$ given by sending each generator to $1$ is quadratically distorted in $H$. Thus, $G$ contains a distorted subgroup.

Finally, suppose that $\psi$ is not polynomially growing. Then by work of Dahmani--Li \cite[Theorem 4]{DL22} (see also work of Gautero--Lustig \cite{GL08} and Ghosh \cite{Gh23}), it is hyperbolic relative to a finite collection of (infinite index) polynomially growing \{fg free\}-by-cyclic subgroups. Since $F_n$ is a finitely generated normal subgroup of a relatively hyperbolic group, it is exponentially distorted by a result of Tran \cite[Corollary 1.2]{Tr21}.
\end{proof}

\begin{proof}[Proof of \cref{thm:locally_undistorted}]
By \cref{thm:locally_rel_qc}, $G$ is locally relatively quasi-convex. By \cite[Theorem 1.4]{Hr10}, the distortion of finitely generated subgroups of $G$ is bounded above by the superadditive closure of the distortion of finitely generated subgroups of the peripheral subgroups. Since the peripherals are all mapping tori of finitely generated free groups, the result now follows by combining \cref{lem:non_surjective_distortion} with \cref{prop:distortion_fg_case}.
\end{proof}

\bibliographystyle{amsalpha}
\bibliography{bibliography}

\end{document}